\documentclass[12pt]{amsart}
\usepackage{amsmath}
\usepackage{amssymb}
\usepackage{amsthm}

\input xy
\xyoption{matrix}
\xyoption{arrow}
\xyoption{2cell}
\UseTwocells

\theoremstyle{plain}
\newtheorem{theorem}{Theorem}[section]
\newtheorem{corollary}[theorem]{Corollary}
\newtheorem{proposition}[theorem]{Proposition}
\newtheorem{lemma}[theorem]{Lemma}

\theoremstyle{definition}
\newtheorem{remark}[theorem]{Remark}
\newtheorem{definition}[theorem]{Definition}
\newtheorem{construction}[theorem]{Construction}

\newcommand{\abs}[1]{\left\vert{#1}\right\vert}
\newcommand{\bd}[1]{\mathbf{#1}}
\newcommand{\br}[1]{\langle#1\rangle}
\newcommand{\C}{\mathcal{C}}
\newcommand{\Cat}{\mathbf{Cat}}
\newcommand{\id}{\mathrm{id}}
\newcommand{\D}{\mathcal{D}}
\newcommand{\E}{\mathcal{E}}
\newcommand{\Fhat}{\hat{F}}
\newcommand{\kappabar}{\overline{\kappa}}
\newcommand{\lambdabar}{\overline{\lambda}}
\newcommand{\multprod}{\Gamma}
\newcommand{\Mult}{\mathbf{Mult}}
\newcommand{\Ob}{{\mathop{\textnormal{Ob}}}}
\newcommand{\Prod}{\prod\limits}
\newcommand{\Sym}{\mathbf{Sym}}
\newcommand{\xhat}{\hat{\bd x}}
\newcommand{\xtilde}{\tilde{\bd x}}

\begin{document}
\title{Multicategories from Symmetric Monoidal Categories}

\author{A.\ D.\ Elmendorf}
\address{Department of Mathematics\\
Purdue University Northwest\\
 Hammond, IN 46323}
 
\email{adelmend@pnw.edu}

\date{July 31, 2025}

\begin{abstract}
This paper considers the possible underlying multicategories for a symmetric monoidal category, and
shows that, up to canonical and coherent isomorphism, there really is only one.  As a result, there is a well-defined forgetful
functor from symmetric monoidal categories to multicategories, as long as all morphisms of symmetric monoidal categories
are at least lax symmetric monoidal.  
\end{abstract}

\maketitle

\tableofcontents

\section{Introduction}

The main aim of this paper is to give a rigorous treatment of the folk theorem that a symmetric monoidal category has a canonical
underlying multicategory.  The problem with this statement is that it is false: a general symmetric monoidal category has many
different underlying multicategories, and there is no truly natural choice of which one to use.  However, even if we use an absurdly
large collection of underlying multicategories, as we will do in this paper, they are all canonically and coherently isomorphic.  As a 
result, it doesn't matter which ones we use, or even if we use different ones for different symmetric monoidal categories: we still
obtain a functor from symmetric monoidal categories (with lax monoidal functors between them) to multicategories, given by
whatever underlying multicategory we choose for each symmetric monoidal category.  

This underlying multicategory functor has a ``weak'' left adjoint, whose construction is the same as the actual left adjoint to the
underlying multicategory functor from permutative categories and strict maps, as constructed in \cite{EM2}.  However, the ``counit''
of the adjunction is only natural with respect to strict maps, which is far too restrictive in general. The adjunction triangles do
however commute, so we do get an adjunction comonad on symmetric monoidal categories, which provides us with a strictification
construction different from that given in \cite{Isbell}.  

Previous work on this issue has concentrated on monoidal categories that need not be symmetric: see, for example, Hermida's work
in \cite{Hermida}, especially section 7, as well as section 3.3 of Leinster's \cite{Leinster}. In particular, Leinster discusses symmetric
multicategories in Section A.2, but doesn't relate them to symmetric monoidal categories, which he doesn't appear to discuss at all.

The author owes an enormous debt of gratitude to the anonymous referee, who scrutinized the paper with astonishing thoroughness
and offered perceptive and probing comments throughout the paper, to its great benefit.  All flaws that remain are, however, entirely
the responsibility of the author.  

It has been a pleasure to discuss some of the material in this paper with Anna Marie Bohmann, Cary Malkiewich, Mona Merling, 
and Inna Zakharevich.  They too are in no way responsible for any errors or omissions that may occur in the paper.

\section{Definitions and Results}

We begin by describing the enormous array of possible underlying multicategories for a symmetric monoidal category, and this
relies on the categorical operad $Y$ described in \cite{SMC1}
whose algebras are precisely symmetric monoidal categories.
The objects of the component category $Y(n)$ can be thought of
as a complete parenthesization of a list of length at least $n$, together with a way of permuting and inserting $n$ variables into
slots in the list.  The rest of the slots are to be filled with identity elements.  
There is a unique morphism from any object of $Y(n)$ to any other, so all diagrams in $Y(n)$ commute.
Then we will describe one underlying multicategory
for a symmetric monoidal category $\C$ for each sequence of functions
\[
\kappa_n:\Ob\C^n\to\Ob Y(n)
\]
 for $n\ge0$, with no further structure.  
 We give an example of such a sequence at the beginning of Section \ref{properties}.
We will write this underlying multicategory as $U_\kappa\C$.
The idea is that for each $n$-tuple $\bd x=(x_1,\dots,x_n)$ of objects of $\C$, the image $\kappa_n\bd x$
tells us how to combine the objects into a single object using the symmetric monoidal structure of $\C$.  Formally, we obtain
an object $\kappa_n(\bd x)(\bd x)$ for each $\bd x\in\C^n$; we will abbreviate this, and write
\[
\kappa_n(\bd x)(\bd x)=\kappabar\bd x.
\]
For notation, we will always write a morphism set in a multicategory with a semicolon separating the source list
from the target object.
Then the formal definition of the morphism sets in $U_\kappa\C$ is almost as follows:

\begin{definition}
Given an $n$-tuple $\bd x=(x_1,\dots,x_n)$ of objects of $\C$ and a single object $y$ of $\C$, we define
\[
U_\kappa\C(\bd x;y)=\C(\kappabar\bd x,y).
\]
\end{definition}

The reason this is almost, but not quite, the correct definition is described at the beginning of Section \ref{structure}.

Our first major theorem is then

\begin{theorem}\label{formsamulticat}
The morphism sets $U_\kappa\C(\bd x;y)$ form a multicategory, with objects the objects of $\C$.
\end{theorem}

The structure is given in Section \ref{structure}; the verification that the properties of a multicategory are
satisfied is somewhat lengthy, and 
deferred until Section \ref{properties}.

Our next major theorem shows that these underlying multicategories are all essentially the same.

\begin{theorem}\label{isomulticats}
Let $\lambda_n:\Ob\C^n\to\Ob Y(n)$ be another sequence of functions defining another underlying
multicategory $U_\lambda\C$.  Then there is a canonical isomorphism of multicategories
\[
U_\kappa\C\to U_\lambda\C.
\]
Further, given a third sequence $\rho_n:\Ob\C^n\to\Ob Y(n)$, the diagram of canonical isomorphisms
\[
\xymatrix{
U_\kappa\C
\ar[rr]
\ar[dr]
&&U_\lambda\C
\ar[dl]
\\&U_\rho\C
}
\]
commutes, so the isomorphisms are both canonical and coherent.
\end{theorem}

The proof is in Section \ref{isos}.

In order for theorem \ref{isomulticats} to give us an actual underlying multicategory functor, we need the morphisms
between symmetric monoidal categories to induce morphisms between the underlying multicategories.  To this end, we have: 

\begin{theorem}\label{laxtheorem}
Let $\C$ and $\D$ be symmetric monoidal categories, with arbitrary choices $U\C$ and $U\D$ of underlying multicategories.
Then a map of multicategories (a \emph{multifunctor})
\[
U\C\to U\D
\]
determines and is determined by a lax symmetric monoidal functor $\C\to\D$.  This assignment produces a fully faithful functor
from symmetric monoidal categories with lax symmetric monoidal functors to multicategories with multifunctors.
\end{theorem}

The proof is in Section \ref{lax}.  As a result,\ Theorem \ref{isomulticats} shows that any choices of underlying multicategories
produce canonically isomorphic underlying multicategory functors.

\begin{corollary}
Any category of symmetric monoidal categories, all of whose morphism functors are at least lax symmetric monoidal, supports
an underlying multicategory functor from symmetric monoidal categories to multicategories.
\end{corollary}

Our last major theorem gives this underlying multicategory functor a weak left adjoint, where ``weak'' means that the counit map
is only lax natural.

\begin{theorem}\label{ladjtheorem}
The forgetful functor from symmetric monoidal categories to multicategories has a weak left adjoint, 
as long as we include strong
symmetric monoidal functors as maps of symmetric monoidal categories.  The adjunction is weak in the sense that
the counit is only natural up to a natural transformation.
\end{theorem}

The construction is exactly the same as the left adjoint constructed in \cite{EM2}, Theorem 4.2, in the context of permutative
categories and strict maps.  
The proof is in Section \ref{ladj}.
Since the unit of the weak adjunction is actually natural, we do get a comonad, although only a weak monad.  We also
observe the following in Section \ref{ladj}:

\begin{theorem}\label{strict}
The weak left adjoint from Theorem \ref{ladjtheorem} converts multicategories and multifunctors into permutative categories
and strict maps, so the comonad of the weak adjunction converts symmetric monoidal categories and lax symmetric monoidal functors
into permutative categories and strict maps.  The (weak) counit of the adjunction is a homotopy equivalence of categories.
\end{theorem}

This is not quite the same as the strictification construction due originally to Isbell \cite{Isbell}, and made explicit by May in
\cite{Einfinity}.  We discuss the comparison and differences in Section \ref{ladj}.

\section{Underlying Multicategories: Structure}\label{structure}
In this section we give the structure of $U_\kappa\C$ as a multicategory, following the definition 
of a multicategory
given in \cite{EM1}, Section 2.  (There are other sources for the definition, but we will use this one for reference.)  The
objects of $U_\kappa\C$ are just the objects of $\C$, and we have already almost defined the morphism sets, namely, given
a list $\bd x=(x_1,\dots,x_n)$ of objects of $\C$ and a target object $y$, we would like to define
\[
U_\kappa\C(\bd x;y)=\C(\kappabar\bd x,y)=\C(\kappa_n(\bd x)(\bd x),y).
\]
However, this definition obscures a subtle technical point that will be important later: if $\bd x\ne\bd x'$, then $U_\kappa\C(\bd x;y)$
must be disjoint from $U_\kappa\C(\bd x';y)$.  For example, if $a,b\in\C$, 
and we choose $\kappabar(a, b)=a\oplus b=\kappabar(a\oplus b)$,
the proposed definition would say that
\[
U_\kappa\C(a,b;a\oplus b)=\C(a\oplus b,a\oplus b)=U_\kappa\C(a\oplus b;a\oplus b),
\]
which makes $\id_{a\oplus b}$ play the role of both a 2-morphism and a 1-morphism.  To avoid this problem rigorously, we would need
to define instead
\[
U_\kappa\C(\bd x;y)=(\bd x,\C(\kappabar\bd x,y))
\]
in order to explicitly keep track of the source string in $U_\kappa\C$.  However, to avoid making notation more cumbersome than strictly
necessary we will continue to drop the explicit source string, and just use $\C(\kappabar\bd x,y)$. When necessary, we will
indicate whether an element is to be considered an element of a particular arity by using a superscript: so in the example, 
$\id_{a\oplus b}^2$ will indicate $\id_{a\oplus b}$ as an element of $U_\kappa\C(a,b;a\oplus b)$, while $\id_{a\oplus b}^1$ will indicate
$\id_{a\oplus b}$ as an element of $U_\kappa\C(a\oplus b;a\oplus b)$.

We must specify identity elements in $U_\kappa\C(a;a)=\C(\kappabar(a),a)$ for each object $a$ of $\C$.  Now $\kappabar(a)
=\kappa_1(a)(a)$, and $\kappa_1(a)$ is an object of $Y(1)$, which also has the canonical identity object $1\in Y(1)$.  We therefore have
a unique isomorphism $\omega(a):\kappa_1(a)\to1$ in $Y(1)$, and applying this to $a$ itself, we get a canonical isomorphism in $\C$:
\[
\omega(a)(a):\kappabar(a)=\kappa_1(a)(a)\to1\cdot a=a,
\]
which we take as our identity element in $U_\kappa\C(a;a)=\C(\kappabar(a),a)$.

We must give a right action of $\Sigma_n$ on the morphism sets with source $n$-tuples of objects. For $\bd x=(x_1,\dots,x_n)$
and $\sigma\in\Sigma_n$,
let's write
\[
\sigma^{-1}\bd x=(x_{\sigma(1)},\dots,x_{\sigma(n)})
\]
in accordance with the standard left action of $\Sigma_n$.
Then we need to define a map
\[
\sigma^*:U_\kappa\C(\bd x;y)\to U_\kappa\C(\sigma^{-1}\bd x;y)
\]
which we will verify does give a right action in Section \ref{properties}.  We define $\sigma^*$ by appealing to the definition
of an operad action to see that
\[
\kappabar\bd x=\kappa_n(\bd x)(\bd x)=(\kappa_n(\bd x)\cdot\sigma)(\sigma^{-1}\bd x).
\]
Since $\kappa_n(\sigma^{-1}\bd x)$ and $\kappa_n(\bd x)\cdot\sigma$ are both objects of $Y(n)$ there is a unique isomorphism
\[
\theta(\bd x,\sigma):\kappa_n(\sigma^{-1}\bd x)\to\kappa_n(\bd x)\cdot\sigma.
\]
As a result, we have a canonical map, which is an isomorphism,
\[
\xymatrix{
\kappabar(\sigma^{-1}\bd x)=\kappa_n(\sigma^{-1}\bd x)(\sigma^{-1}\bd x)
\ar[r]^-{\theta(\bd x,\sigma)}
&\kappa_n(\bd x)\cdot\sigma\sigma^{-1}\bd x=\kappa_n(\bd x)(\bd x)=\kappabar\bd x
}
\]
which induces our desired map
\[
\sigma^*:U_\kappa\C(\bd x;y)=\C(\kappabar\bd x,y)\to\C(\kappabar(\sigma^{-1}\bd x),y)=U_\kappa\C(\sigma^{-1}\bd x;y).
\]
Explicitly, we can say $\sigma^*=\C(\theta(\bd x,\sigma),1)$.

Finally, we must specify a multiproduct that gives the composition in $U_\kappa\C$, and in order to do so efficiently, we need
some notation.  Suppose we are given a final target object $z$, a tuple $\bd y=(y_1,\dots,y_n)$ that we will map to $z$, and for each $s$ with $1\le s\le n$,
a tuple $\bd x_s=(x_{s1},\dots,x_{sj_s})$ that we will map to $y_s$.  Let $j=j_1+\cdots+j_n$, and write $\odot_s\bd x_s$ for 
the $j$-tuple that is
the concatenation
of all the $\bd x_s$'s.  Then we must specify a composition multiproduct
\[
\Gamma:U_\kappa\C(\bd y;z)\times\prod_{s=1}^n U_\kappa\C(\bd x_s;y_s)\to U_\kappa\C(\odot_s\bd x_s;z).
\]
But using the definition of $U_\kappa\C$, this means giving a map
\[
\Gamma:\C(\kappabar\bd y,z)\times\prod_{s=1}^n\C(\kappabar\bd x_s,y_s)\to\C(\kappabar(\odot_s\bd x_s),z).
\]
Now we need just a bit more notation.  
First, 
given a list of items such as $\kappabar\bd x_1,\dots,\kappabar\bd x_n$, we write the
entire list as $\br{\kappabar\bd x_s}$, with the subscript presumed to run over appropriate limits.  Next,
since we are using $\Gamma$ for our composition in $U_\kappa\C$, we will use $\gamma$ for the
operad operation in $Y$ to avoid confusion.  (We used $\Gamma$ in \cite{SMC1} for the operation in $Y$.)

We can now rewrite
\[
\prod_{s=1}^n\C(\kappabar\bd x_s,y_s)=\C^n(\br{\kappabar\bd x_s},\bd y),
\]
and noting that since $\kappa_n(\bd y)$ is an object of $Y(n)$, it induces a functor $\C^n\to\C$.  
Further, both $\kappa_j(\odot_s\bd x_s)$ and $\gamma(\kappa_n\bd y;\br{\kappa_{j_s}\bd x_s})$ are objects
of $Y(j)$, and so there is a unique isomorphism
\[
\phi(\bd y,\br{\bd x_s}):
\kappa_j(\odot_s\bd x_s)\to\gamma(\kappa_n\bd y;\br{\kappa_{j_s}\bd x_s}),
\]
which we can apply to the object $\odot_s\bd x_s$ to obtain an isomorphism which we abusively denote
with the same notation:
\[
\phi(\bd y,\br{\bd x_s}):
\kappabar(\odot_s\bd x_s)
\to
\gamma(\kappa_n\bd y;\br{\kappa_{j_s}\bd x_s})(\odot_s\bd x_s)
=\kappa_n\bd y\br{\kappabar\bd x_s},
\]
and consequently an induced map
\[
\C(\phi(\bd y,\br{\bd x_s}),1):
\C(\kappa_n(\bd y)\br{\kappabar\bd x_s},z)
\to\C(\kappabar(\odot_s\bd x_s),z).
\]

We now define our composition multiproduct
in $U_\kappa\C$ as the following composite:
\[
\xymatrix{
\C(\kappabar\bd y,z)\times\prod_{s=1}^n\C(\kappabar\bd x_s,y_s)
\ar[d]^-{=}
\\\C(\kappabar\bd y,z)\times\C^n(\br{\kappabar\bd x_s},\bd y)
\ar[d]^-{1\times\kappa_n(\bd y)}
\\\C(\kappabar\bd y,z)\times\C(\kappa_n(\bd y)\br{\kappabar\bd x_s},\kappabar\bd y)
\ar[d]^-{\circ}
\\\C(\kappa_n(\bd y)\br{\kappabar\bd x_s},z),
\ar[d]^-{\C(\phi(\bd y,\br{\bd x_s}),1)}
\\\C(\kappabar(\odot_s\bd x_s),z).
}
\]
This completes the specification of the structure of $U_\kappa\C$ as a multicategory.
The verification that this structure satisfies the properties of a multicategory is deferred to Section \ref{properties}.

\section{Isomorphisms of Underlying Multicategories}\label{isos}

In this section we 
prove Theorem \ref{isomulticats}, which
shows that all the multicategories $U_\kappa\C$ for all possible values of $\kappa$
are canonically and coherently isomorphic.  This means that we can reasonably speak of ``the'' underlying
multicategory for a symmetric monoidal category, since any one is unique up to a unique canonical
isomorphism.  

We suppose given two arbitrary sequences 
\[
\kappa_n:\Ob\C^n\to\Ob Y(n)\text{ and }\lambda_n:\Ob\C\to\Ob Y(n),
\]
which determine underlying multicategories $U_\kappa\C$ and $U_\lambda\C$.
The proof of the theorem consists of the definition of the canonical isomorphism $U_\kappa\C\to U_\lambda\C$,
the verification that the coherence diagram commutes, and finally 
the verification that our map is an isomorphism of multicategories.  
We begin with the definition of the canonical isomorphism.

\begin{definition}
Let $\kappa_n:\Ob\C^n\to \Ob Y(n)$ and $\lambda_n:\Ob\C^n\to \Ob Y(n)$ be any two sequences of functions.  Then
given an object $\bd x\in\C^n$, we have the objects $\kappa_n\bd x$ and $\lambda_n\bd x$ of $Y(n)$, and
we define
\[
\alpha(\bd x):\lambda_n\bd x\to\kappa_n\bd x
\]
to be the unique isomorphism in $Y(n)$ with the given source and target. 

To define an isomorphism of multicategories $U_\kappa\C\to U_\lambda\C$, we must give maps on objects
and on sets of morphisms.  On objects, we just use $\id_{\Ob\C}$, since both underlying multicategories
have $\Ob\C$ as their objects.  On morphisms, for each $\bd x\in\C^n$, we have the induced map, which technically
should be labeled $\alpha(\bd x)(\bd x)$,
\[
\xymatrix{
\lambdabar\bd x=\lambda_n(\bd x)(\bd x)
\ar[r]^-{\alpha(\bd x)}
&\kappa_n(\bd x)(\bd x)=\kappabar\bd x,
}
\]
which in turn induces our desired map
\[
\xymatrix@C+20pt{
U_\kappa(\bd x;y)=\C(\kappabar\bd x,y)
\ar[r]^-{\C(\alpha(\bd x),1)}
&\C(\lambdabar\bd x,y)=U_\lambda(\bd x;y).
}
\]
\end{definition}

It follows immediately from the definition that these maps are coherent, in the sense given in the theorem and
in greater detail in the following corollary.

\begin{corollary}
Let $\rho_n:\Ob\C^n\to\Ob Y(n)$ be a third set of functions defining an underlying multicategory $U_\rho\C$, let
\[
\beta(\bd x):\rho_n\bd x\to\lambda_n\bd x\text{ and }\delta(\bd x):\rho_n\bd x\to\kappa_n\bd x
\]
be the unique isomorphisms in $Y(n)$ inducing the alleged isomorphisms of underlying multicategories $U_\lambda\C\to U_\rho\C$
and $U_\rho\C\to U_\kappa\C$.  Then the diagram of induced isomorphisms
\[
\xymatrix{
U_\kappa\C
\ar[rr]^-{\cong}
\ar[dr]_-{\cong}
&&U_\lambda\C
\ar[dl]^-{\cong}
\\&U_\rho\C
}
\]
commutes.  Further, the induced automorphism on any one underlying multicategory is the identity.
\end{corollary}

\begin{proof}
The isomorphisms of underlying categories are induced by the diagrams of isomorphisms in $Y(n)$ 
\[
\xymatrix{
\kappa_n\bd x
&&\lambda_n\bd x
\ar[ll]_-{\alpha(\bd x)}
\\&\rho_n\bd x,
\ar[ul]^-{\delta(\bd x)}
\ar[ur]_-{\beta(\bd x)}
}
\]
which commute since they are diagrams in $Y(n)$, where all diagrams commute.  If we set $\lambda=\kappa$, then 
$\alpha(\bd x)=\id_{\kappa_n\bd x}$, so the automorphism of $U_\kappa\C$ is the identity.
\end{proof}

We must show that this definition actually preserves multicategory structure.  Since the inducing maps $\alpha(\bd x)$
are all isomorphisms, the induced maps on morphism sets are all bijections, so this will show that we do have
an isomorphism of multicategories.  We begin with the identity structure.  

\begin{proposition}
The identity element $\id_a\in U_\kappa\C(a;a)$ is sent to $\id_a\in U_\lambda\C(a;a)$.
\end{proposition}

\begin{proof}
We have the unique isomorphisms in $Y(1)$
\[
\omega(a):\kappa_1(a)\to 1\text{ and }\omega'(a):\lambda_1(a)\to1
\]
which applied to $a$
give us the identity elements in $U_\kappa\C(a;a)=\C(\kappabar a,a)$ and $U_\lambda\C(a;a)=\C(\lambdabar a,a)$.
But we also have the diagram
\[
\xymatrix{
\lambda_1a
\ar[rr]^-{\alpha(a)}
\ar[dr]_-{\omega'(a)}
&&\kappa_1a
\ar[dl]^-{\omega(a)}
\\&1
}
\]
in $Y(1)$, where all diagrams commute, and so
\[
\xymatrix@C+15pt{
U_\kappa\C(a;a)=\C(\kappabar a,a)
\ar[r]^-{\C(\alpha(a),1)}
&\C(\lambdabar a,a)=U_\lambda\C(a;a)
}
\]
sends $\omega(a)$ to $\omega'(a)$, and therefore preserves the identity morphisms.
\end{proof}

Next, we show that composition is preserved.
Following the definition of the composition in $U_\kappa\C$ given in Section \ref{structure}, we assume given a final target object $z$,
an $n$-tuple $\bd y=(y_1,\dots,y_n)$ mapping to $z$, and for each $1\le s\le n$, a $j_s$-tuple $\bd x_s$ that
will map to $y_s$.  Also as before, we say $j=j_1+\cdots j_n$, and the concatenation $\odot_s\bd x_s$
of all the $\bd x_s$'s is therefore a $j$-tuple.

\begin{proposition}
The maps of morphism sets defined above preserve composition, meaning that the following diagram commutes:
\[
\xymatrix{
\C(\kappabar\bd y,z)\times\C^n(\br{\kappabar\bd x_s},\bd y)
\ar[r]^-{\Gamma}
\ar[d]_-{\C(\alpha(\bd y),1)\times\C^n(\br{\alpha(\bd x_s)},1)}
&\C(\kappabar(\odot_s\bd x_s),z)
\ar[d]^-{\C(\alpha(\odot_s\bd x_s),1)}
\\\C(\lambdabar\bd y,z)\times\C^n(\br{\lambdabar\bd x_s},\bd y)
\ar[r]_-{\Gamma}
&\C(\lambdabar(\odot_s\bd x_s),z).
}
\]
\end{proposition}

\begin{proof}
We expand the diagram as follows using the definition of $\Gamma$, where $g$ is induced by the unique isomorphism
in $Y(j)$ given by
\[
\gamma(\lambda_n\bd y;\br{\lambda_{j_s}\bd x_s})\to\gamma(\kappa_n\bd y;\br{\kappa_{j_s}\bd x_s}),
\]
and $\phi'(\bd y,\br{\bd x_s}):\lambdabar(\odot_s\bd x_s)\to\lambda_n\br{\lambdabar\bd x_s}$ is the $\lambda$-analogue
of $\phi(\bd y,\br{\bd x_s})$:
\[
\xymatrix@C+50pt{
\C(\kappabar\bd y,z)\times\C^n(\br{\kappabar\bd x_s},\bd y)
\ar[r]^-{1\times\kappa_n\bd y}
\ar[d]_-{\C(\alpha(\bd y),1)\times\C^n(\br{\alpha(\bd x_s)},1)}
&\C(\kappabar\bd y, z)\times\C(\kappa_n\bd y\br{\kappabar\bd x_s},\kappabar\bd y)
\ar[ddl]^-{\,\,\,\,\,\,\,\,\,\,\C(\alpha(\bd y),1)\times\C(g,\alpha(\bd y)^{-1})}
\ar[d]^-{\circ}
\\\C(\lambdabar\bd y,z)\times\C^n(\br{\lambdabar\bd x_s},\bd y)
\ar[d]_-{1\times\lambda_n\bd y}
&\C(\kappa_n\bd y\br{\kappabar\bd x_s},z)
\ar[ddl]^-{\C(g,1)}
\ar[d]^-{\C(\phi(\bd y,\br{\bd x_s}),1)}
\\\C(\lambdabar\bd y,z)\times\C(\lambda_n\bd y\br{\lambdabar\bd x_s},\lambdabar\bd y)
\ar[d]_-{\circ}
&\C(\kappabar(\odot_s\bd x_s),z)
\ar[d]^-{\C(\alpha(\odot_s\bd x_s),1)}
\\\C(\lambda_n\bd y\br{\lambdabar\bd x_s},z)
\ar[r]_-{\C(\phi'(\bd y,\br{\bd x_s}),1)}
&\C(\lambdabar(\odot_s\bd x_s),z).
}
\]
The bottom (distorted) square commutes because all the maps are induced by isomorphisms in $Y(j)$, where all
diagrams commute.  The middle (somewhat less distorted) square commutes because we are composing with
both $\alpha(\bd y)$ and its inverse, which then cancel.  The top distorted square 
is the product of two separate squares, and
commutes when restricted to the
first factor $\C(\kappabar\bd y,z)$ by inspection.  This reduces the argument to verifying that the top square commutes when
restricted to the second factor, at which point the desired diagram becomes
\[
\xymatrix{
\C^n(\br{\kappabar\bd x_s},\bd y)
\ar[r]^-{\kappa_n\bd y}
\ar[d]_-{\C^n(\br{\alpha(\bd x_s)},1)}
&\C(\kappa_n\bd y\br{\kappabar\bd x_s},\kappabar\bd y)
\ar[d]^-{\C(g,\alpha(\bd y)^{-1})}
\\\C^n(\br{\lambdabar\bd x_s},\bd y)
\ar[r]_-{\lambda_n\bd y}
&\C(\lambda_n\bd y\br{\lambdabar\bd x_s},\lambdabar\bd y).
}
\]
Tracing a typical element $\br{f_s}\in\C^n(\br{\kappabar\bd x_s},\bd y)$ through the diagram, we find that
commutativity requires us to verify that
\[
\lambda_n\bd y\br{f_s\circ\alpha(\bd x_s)}=\alpha(\bd y)^{-1}\circ\kappa_n\bd y\br{f_s}\circ g,
\]
or equivalently,
\[
\alpha(\bd y)\circ\lambda_n\bd y\br{f_s\circ\alpha(\bd x_s)}=\kappa_n\bd y\br{f_s}\circ g,
\]
where we recall that 
$g$ is induced by the unique map in $Y(j)$ 
\[
\gamma(\lambda_n\bd y;\br{\lambda_{j_s}\bd x_s})\to\gamma(\kappa_n\bd y;\br{\kappa_{j_s}\bd x_s}),
\]
therefore giving the canonical isomorphism
\[
\lambda_n\bd y\br{\lambdabar\bd x_s}\to\kappa_n\bd y\br{\kappabar\bd x_s}.
\]
Our desired equality now becomes the commutativity of the following diagram, in which the top row composes to $g$:
\[
\xymatrix@C+20pt{
\lambda_n\bd y\br{\lambdabar\bd x_s}\ar[r]^-{\alpha(\bd y)}
\ar[d]_-{\lambda_n\bd y\br{f_s\circ\alpha(\bd x_s)}}
&\kappa_n\bd y\br{\lambdabar\bd x_s}
\ar[r]^-{\kappa_n\bd y\br{\alpha(\bd x_s)}}
\ar[dr]_-{\kappa_n\bd y\br{f_s\circ\alpha(\bd x_s)}\,\,\,\,\,\,\,\,\,\,}
&\kappa_n\bd y\br{\kappabar\bd x_s}
\ar[d]^-{\kappa_n\bd y\br{f_s}}
\\\lambdabar\bd y
\ar[rr]_-{\alpha(\bd y)}
&&\kappabar\bd y.
}
\]
The top row does compose to $g$, since it is induced by maps in $Y(j)$, where all diagrams commute.  The triangle
commutes by functoriality of $\kappa_n\bd y$, and the left part of the diagram by naturality of $\alpha(\bd y)$.
The total diagram therefore commutes, which completes the proof that our map of multicategories preserves composition.
\end{proof}

To conclude showing that we have a map of multicategories, we must show that the $\Sigma_n$-actions on
the morphism sets are preserved.  This is the content of the following proposition:

\begin{proposition}
Let $\bd x\in\C^n$, $y\in\C$, and $\sigma\in\Sigma_n$.  Then the following diagram commutes:
\[
\xymatrix@C+30pt{
U_\kappa\C(\bd x;y)
\ar[r]^-{\C(\alpha(\bd x),1)}
\ar[d]_-{\sigma^*}
&U_\lambda\C(\bd x;y)
\ar[d]^-{\sigma^*}
\\U_\kappa\C(\sigma^{-1}\bd x;y)
\ar[r]_-{\C(\alpha(\sigma^{-1}\bd x),1)}
&U_\lambda\C(\sigma^{-1}\bd x;y).
}
\]
\end{proposition}

\begin{proof}
We already have the unique isomorphism $\theta(\bd x,\sigma):\kappa_n(\sigma^{-1}\bd x)\to\kappa_n\bd x\cdot\sigma$ in $Y(n)$ which 
induces the map $\sigma^*$ in $U_\kappa\C$, and analogously let 
\[
\pi(\bd x,\sigma):\lambda_n(\sigma^{-1}\bd x)\to\lambda_n\bd x\cdot\sigma
\]
be the unique isomorphism in $Y(n)$ inducing the map $\sigma^*$ in $U_\lambda\C$.
Then the required square converts and expands as follows:
\[
\xymatrix@C+40pt{
\C(\kappabar\bd x,y)
\ar[r]^-{\C(\alpha(\bd x),1)}
\ar[d]_-{=}
&\C(\lambdabar\bd x,y)
\ar[d]^-{=}
\\\C((\kappa_n\bd x\cdot\sigma)(\sigma^{-1}\bd x),y)
\ar[d]_-{\C(\theta(\bd x,\sigma),1)}
\ar[r]^-{\C(\alpha(\bd x)\cdot\sigma,1)}
&\C((\lambda_n\bd x\cdot\sigma)(\sigma^{-1}\bd x),y)
\ar[d]^-{\C(\pi(\bd x,\sigma),1)}
\\\C(\kappabar(\sigma^{-1}\bd x),y)
\ar[r]_-{\C(\alpha(\sigma^{-1}\bd x),1)}
&\C(\lambdabar(\sigma^{-1}\bd x),y).
}
\]
The top square commutes because both horizontal arrows express the map induced by $\alpha(\bd x)$.  The bottom square
commutes since all the arrows are induced by maps in $Y(n)$, where all diagrams commute.  The induced map 
therefore preserves the $\Sigma_n$-action.  This concludes the proof that we have defined a map of multicategories, which
must be an isomorphism since it is a bijection on objects and all morphism sets, and further the inverse is induced by the maps
$\alpha(\bd x)^{-1}$.
\end{proof}

\section{The Relation to Lax Symmetric Monoidal Functors}\label{lax}

We have now shown that all of the vast collection of possible underlying multicategories $U_\kappa\C$ for $\C$ are canonically
and coherently isomorphic, so we obtain of an underlying multicategory functor that is unique up to unique isomorphism.  This still leaves open
the issue of what sort of maps of symmetric monoidal categories can be used to give maps of underlying multicategories.  This
is answered by Theorem \ref{laxtheorem}, whose proof occupies this section.

Before starting the proof of the theorem, it will be convenient to make some assumptions
about the sequences $\{\kappa_n\}$ defining our underlying multicategories.  
This is justified by the fact that all the underlying multicategories we have defined are canonically and coherently isomorphic, so
we can choose any one of them without loss of generality.
In particular, throughout this proof, we will assume that both underlying multicategories
are defined by sequences 
of functions
that are constant on objects of the same length, so are defined by a sequence of particular objects of $Y(n)$
for each $n$, which we will call just $\kappa_n$. We ask that $\kappa_0=0$, the generator in dimension 0 of the objects of $Y$,
that $\kappa_1=1$, the identity element in $Y(1)$, and that $\kappa_2=m$, the generator in dimension 2 of the objects of $Y$; note
that the action of $Y$ on any symmetric monoidal category $\C$ sends $0$ to the unit object $e_\C$, and sends $m$
to the monoidal product, so $m\cdot(a,b)=a\oplus b$.  For larger values of $n$, we use induction to define
\[
\kappa_n=\gamma(m;\kappa_{n-1},1).
\]
Note however that $\kappa_1\ne\gamma(m;\kappa_0,1)$, since $\gamma(m;0,1)$ acts on an object $a$
to produce $e_\C\oplus a$, while $\kappa_1=1$ acts as the identity.

The effect of our assumed values of the $\kappa_n$'s is that we are parenthesizing all products by piling up the parentheses to the left.

We will drop the argument $\bd x$ of $\kappa_n\bd x$ for an object $\bd x\in\C^n$, since we are assuming $\kappa_n$
is constant, and instead just write $\kappa_n$.

As an additional notational assumption, we suppose give a multifunctor $\Fhat$ and wish to produce a lax symmetric monoidal
functor $F$.  Putting the hat on the multifunctor will distinguish it from its induced monoidal functor, while reducing the number of hats
in the description.

Now we begin the proof by assuming that we are given a multifunctor $\Fhat:U_\kappa\C\to U_\kappa\D$; we must show that this induces a lax
symmetric monoidal functor $F:\C\to\D$.  In particular, we must produce an induced functor $F:\C\to\D$, and show that we have
induced natural transformations $\eta:e_\D\to Fe_\C$ and $\xi:Fa\oplus Fb\to F(a\oplus b)$, subject to the following three diagrams,
where $c_a:a\oplus e_\C\to a$ is the unit isomorphism,
$\tau:a\oplus b\to b\oplus a$ is the commutativity isomorphism, and
$\alpha:(a\oplus b)\oplus c\to a\oplus(b\oplus c)$ is the associativity isomorphism:
\[
\xymatrix{
Fa\oplus e_\D
\ar[r]^-{1\oplus\eta}
\ar[d]_-{c_{Fa}}
&Fa\oplus Fe_\C
\ar[d]^-{\xi}
\\F(a)
&F(a\oplus e_\C),
\ar[l]^-{Fc_a}
}
\]
\[
\xymatrix{
(Fa\oplus Fb)\oplus Fc
\ar[r]^-{\alpha}
\ar[d]_-{\xi\oplus1}
&Fa\oplus(Fb\oplus Fc)
\ar[d]^-{1\oplus\xi}
\\F(a\oplus b)\oplus Fc
\ar[d]_-{\xi}
&Fa\oplus F(b\oplus c)
\ar[d]^-{\xi}
\\F((a\oplus b)\oplus c)
\ar[r]_-{F\alpha}
&F(a\oplus(b\oplus c)),
}
\]
and
\[
\xymatrix{
Fa\oplus Fb
\ar[r]^-{\tau}
\ar[d]_-{\xi}
&Fb\oplus Fa
\ar[d]^-{\xi}
\\F(a\oplus b)
\ar[r]_-{F\tau}
&F(b\oplus a).
}
\]

We produce our lax monoidal functor as follows.

\begin{definition}
The functor $F$ is just the underlying functor of the multifunctor.  Since we are assuming $\kappa_1=1$, the identity for $Y(1)$, it follows that
$\kappabar a=a$ for a single object $a$, and therefore
\[
U_\kappa\C(a;b)=\C(\kappabar a,b)=\C(a,b),
\]
and similarly for $U_\kappa\D$.  We do therefore get a functor $F:\C\to\D$ by restricting to 1-morphisms.

The unit map $\eta:e_\D\to Fe_\C$ arises from the map of 0-morphisms.  Since we are assuming $\kappa_0=0\in Y(0)$, and since $\C^0$
is the terminal category with one object which we denote $*$, 
and further the action of $0\in Y(0)$ by definition selects out the unit element in a symmetric monoidal category,
we see that 
\[
U_\kappa\C(;y)=\C(\kappabar(*),y)=\C(0\cdot*,y)=\C(e_\C,y),
\]
by the definition of the action of $Y$ on the symmetric monoidal category $\C$,
and similarly for $U_\kappa\D$. Now given a multifunctor $\Fhat$, we restrict to 0-morphisms and have in particular a map
\[
\xymatrix{
\C(e_\C,e_\C)=U_\kappa\C(;e_\C)
\ar[r]^-{\Fhat_0}
& U_\kappa\D(;\Fhat e_\C)=\D(e_\D,Fe_\C).
}
\]
We define $\eta:e_\D\to Fe_\C$ to be the image of $\id_{e_\C}$ (technically $\id_{e_\C}^0$) under this map of 0-morphisms.

We next define the structure map $\xi:Fa\oplus Fb\to F(a\oplus b)$. Since the action of $Y$ on a symmetric monoidal category
sends the generator $m\in Y(2)$ to the product map, so $m\cdot(a,b)=a\oplus b$, and we have assumed $\kappa_2$ is
constant at the element $m\in Y(2)$, we can restrict our multifunctor to 2-morphisms, and have
\[
\xymatrix{
\C(a\oplus b,a\oplus b)
\ar[d]^-{=}
\\U_\kappa\C(a,b;a\oplus b)
\ar[d]^-{\Fhat_2}
\\U_\kappa\D(\Fhat a,\Fhat b;\Fhat(a\oplus b))
\ar[d]^-{=}
\\\D(Fa\oplus Fb,F(a\oplus b)),
}
\]
and we take $\xi$ to be the image of $\id_{a\oplus b}$ (technically $\id_{a\oplus b}^2$) under this map.
\end{definition}

We must verify the three coherence diagrams, and begin with the unit coherence diagram.

\begin{lemma}
With the above definitions, the diagram
\[
\xymatrix{
Fa\oplus e_\D
\ar[r]^-{1\oplus\eta}
\ar[d]_-{c_{Fa}}
&Fa\oplus Fe_\C
\ar[d]^-{\xi}
\\F(a)
&F(a\oplus e_\C).
\ar[l]^-{Fc_a}
}
\]
commutes.
\end{lemma}

\begin{proof}
This will follow from the definition of the multicomposition in $U_\kappa\C$ and $U_\kappa\D$, and the multifunctoriality of $\Fhat$.
We observe first that we can write $a\oplus e_\C$ as $\gamma(m;1,0)\cdot a$, and then the unit isomorphism
$c_a:a\oplus e_\C\to a$ is induced by the unique isomorphism $\gamma(m;1,0)\to1$ in $Y(1)$.  
Further, this is the inverse of the following isomorphism, which is a special case of the isomorphism $\phi(\bf y,\br{\bf x_s})$
introduced in the definition of the multiproduct on $U_\kappa\C$:
\[
\phi((a,e_\C),a):1=\kappa_1(a)\to\gamma(\kappa_2(a,e_\C);\kappa_1a,\kappa_0(*))=\gamma(m;1,0).
\]
The key step is to examine
the following composition in the multicategory $U_\kappa\C$:
\[
U_\kappa\C(a,e_\C;a\oplus e_\C)\times U_\kappa\C(a;a)\times U_\kappa\C(;e_\C)\to U_\kappa\C(a;a\oplus e_\C).
\]
In particular, we look at the triple $(\id_{a\oplus e_\C}^2,\id_a^1,\id_{e_\C}^0)$ and find its image in the expanded version
of the composition as follows:
\[
\xymatrix{
\C(m\cdot(a,e_\C),a\oplus e_\C)\times\C^2((a,e_\C),(a,e_\C))
\ar[d]^-{1\times m}
\\\C(m\cdot(a,e_\C),a\oplus e_\C)\times\C(m\cdot(a,e_\C),m\cdot(a,e_\C))
\ar[d]^-{\circ}
\\\C(m\cdot(a,e_\C),a\oplus e_\C)
\ar[d]^-{=}
\\\C(\gamma(m;1,0)\cdot a,a\oplus e_\C)
\ar[d]^-{\C(\phi((a,e_\C),a),1)}
\\\C(a,a\oplus e_\C).
}
\]
Tracing the triple $(\id_{a\oplus e_\C},\id_a,\id_{e_\C})$ through the composition, we see that at the next-to-the-last step,
we have $\id_{a\oplus e_\C}^1$.  This is then composed with $\phi((a,e_\C),a)$, which is the inverse of 
the map inducing
$c_a$, so the composite
is $c_a^{-1}$.

Now we apply the multifunctor $\Fhat$ to the entire composite, so by multifunctoriality, we must end up with $\Fhat(c_a^{-1})=(\Fhat c_a)^{-1}$.
The starting point is by definition the triple $(\xi,\id_{Fa},\eta)$, and we trace this through the composite
\[
U_\kappa\D(Fa,Fe_\C;F(a\oplus e_\C))\times U_\kappa\D(Fa;Fa)\times U_\kappa\D(;Fe_\C)\to U_\kappa\D(Fa;F(a\oplus e_\C)),
\]
which expands to
\[
\xymatrix{
\D(m\cdot(Fa,Fe_\C),F(a\oplus e_\C))\times\D^2((Fa,e_\D),(Fa,Fe_\C))
\ar[d]^-{1\times m}
\\\D(m\cdot(Fa,Fe_\C),F(a\oplus e_\C))\times\D(m\cdot(Fa,e_\D),m\cdot(Fa,Fe_\C))
\ar[d]^-{\circ}
\\\D(m\cdot(Fa,e_\D),F(a\oplus e_\C))
\ar[d]^-{\D(\phi((Fa,e_\D),Fa),1)}
\\\D(Fa,F(a\oplus e_\C)).
}
\]
Our triple now traces as follows, since the last map is, as before, composition with $c_{Fa}^{-1}$:
\[
(\xi,\id_{Fa},\eta)\mapsto(\xi,\id_{Fa}\oplus\eta)
\mapsto\xi\circ(\id_{Fa}\oplus\eta)
\mapsto\xi\circ(\id_{Fa}\oplus\eta)\circ c_{Fa}^{-1}.
\]
Since this must coincide with $(\Fhat c_a)^{-1}=(Fc_a)^{-1}$ by multifunctoriality, our desired diagram does commute.
\end{proof}

We must verify commutativity of the associativity coherence diagram, and claim:

\begin{lemma}
The associativity coherence diagram
\[
\xymatrix{
(Fa\oplus Fb)\oplus Fc
\ar[r]^-{\alpha}
\ar[d]_-{\xi\oplus1}
&Fa\oplus(Fb\oplus Fc)
\ar[d]^-{1\oplus\xi}
\\F(a\oplus b)\oplus Fc
\ar[d]_-{\xi}
&Fa\oplus F(b\oplus c)
\ar[d]^-{\xi}
\\F((a\oplus b)\oplus c)
\ar[r]_-{F\alpha}
&F(a\oplus(b\oplus c)).
}
\]
commutes.
\end{lemma}

\begin{proof}
Again we exploit the definition of the multicomposition in the two underlying multicategories, along with the multifunctoriality of $\Fhat$.
We note first that by our convention, $\kappa_3=\gamma(m;m,1)$, and therefore $\kappa_3(a,b,c)=(a\oplus b)\oplus c$.
Further, the associativity isomorphism $(a\oplus b)\oplus c\to a\oplus(b\oplus c)$ is induced by the unique isomorphism in $Y(3)$
\[
\alpha:\kappa_3=\gamma(m;m,1)\to\gamma(m;1,m).
\]

We will need this associativity isomorphism in two different guises: as a 3-morphism in $U_\kappa\C(a,b,c;a\oplus(b\oplus c))$ and
as a 1-morphism in $U_\kappa\C((a\oplus b)\oplus c;a\oplus(b\oplus c))$.  Let's call the first of these $\alpha_3$, and the second one $\alpha_1$.

We begin the argument by showing that the left vertical column in the desired diagram arises from $\id_{(a\oplus b)\oplus c}^3$
by applying the multifunctor $\Fhat$.
We consider the composition in $U_\kappa\C$ as follows:
\[
U_\kappa\C(a\oplus b,c;(a\oplus b)\oplus c)\times U_\kappa\C(a,b;a\oplus b)\times U_\kappa\C(c;c)
\to U_\kappa\C(a,b,c;(a\oplus b)\oplus c),
\]
and trace the image of the triple $(\id_{(a\oplus b)\oplus c}^2,\id_{a\oplus b}^2,\id_c^1)$ through its expansion:
\[
\xymatrix{
\C(m\cdot(a\oplus b,c),(a\oplus b)\oplus c)\times\C^2((m\cdot(a,b),c),(a\oplus b),c)
\ar[d]^-{1\times m}
\\\makebox[.7\textwidth][c]{$C(m\cdot(a\oplus b,c),(a\oplus b)\oplus c)\times\C(m\cdot(m\cdot(a,b),c),m\cdot(a\oplus b,c))$}
\ar[d]^-{\circ}
\\\C(m\cdot(m\cdot(a,b),c),(a\oplus b)\oplus c)
\ar[d]^-{\C(\phi((a\oplus b,c),((a, b), c)),1)}
\\\C((a\oplus b)\oplus c,(a\oplus b)\oplus c).
}
\]
Now 
in the last map of this display, we have
$\phi((a\oplus b,c),((a, b), c)):\kappa_3\to\gamma(\kappa_2;\kappa_2,\kappa_1)$, but the target here is just
\[
\kappa_3=\gamma(m;m,1),
\]
so the last map is an identity, and our triple ends up at $\id_{(a\oplus b)\oplus c}^3$.  Applying the multifunctor $\Fhat$ throughout
will therefore end us up at $\Fhat(\id^3_{(a\oplus b)\oplus c})$.  When we do so, we are looking at 
the image of the triple
\[
(\Fhat(\id_{(a\oplus b)\oplus c}^2),\Fhat(\id_{a\oplus b}^2),\Fhat(\id_c^1))=(\xi_{a\oplus b,c},\xi_{a,b},\id_{\Fhat c})
\]
in
the composition in $U_\kappa\D$
given by
\[
\xymatrix{
\makebox[.7\textwidth][c]{$U_\kappa\D(\Fhat(a\oplus b),\Fhat c;\Fhat((a\oplus b)\oplus c))\times U_\kappa\D(\Fhat a,\Fhat b;\Fhat(a\oplus b))\times U_\kappa\D(\Fhat c,\Fhat c)$}
\ar[d]^-{\Gamma}
\\U_\kappa\D(\Fhat a,\Fhat b,\Fhat c;\Fhat((a\oplus b)\oplus c)).
}
\]
This expands, using the definition of $\Gamma$, as follows:
\[
\xymatrix{
\makebox[.7\textwidth][c]{$\D(m\cdot(F(a\oplus b),Fc),F((a\oplus b)\oplus c))\times\D^2((m\cdot(Fa,Fb),Fc),(F(a\oplus b),Fc))$}
\ar[d]^-{1\times m}
\\\makebox[.7\textwidth][c]{$\D(m\cdot(F(a\oplus b),Fc),F((a\oplus b)\oplus c))\times\D(m\cdot(m\cdot(Fa,Fb),Fc),m\cdot(F(a\oplus b),Fc))$}
\ar[d]^-{\circ}
\\\D(m\cdot(m\cdot(Fa,Fb),Fc),F((a\oplus b)\oplus c))
\ar[d]^-{\D(\phi((F(a\oplus b),Fc),(Fa,Fb,Fc)),1)}
\\\D((Fa\oplus Fb)\oplus Fc,F((a\oplus b)\oplus c)).
}
\]
But again, the final map here is just the identity, since $\kappa_3=\gamma(m;m,1)$.  Our source triple 
therefore traces as
\[
(\xi_{a\oplus b,c},\xi_{a,b},\id_{Fc})
\mapsto
(\xi,\xi\oplus1)
\mapsto
\xi\circ(\xi\oplus1).
\]
We may conclude that $\Fhat$ sends 
\[
\id^3_{(a\oplus b)\oplus c}\in U_\kappa\C(a,b,c;(a\oplus b)\oplus c)
\]
to 
\begin{gather*}
\xi\circ(\xi\oplus 1)\in U_\kappa\D(\Fhat a,\Fhat b,\Fhat c;\Fhat((a\oplus b)\oplus c))
\\=\D((Fa\oplus Fb)\oplus Fc,F((a\oplus b)\oplus c)),
\end{gather*}
which is the left column of the desired diagram.

Next, we consider the composition in $U_\kappa\C$ 
\[
U_\kappa\C((a\oplus b)\oplus c;a\oplus(b\oplus c))\times U_\kappa\C(a,b,c;(a\oplus b)\oplus c)
\to
U_\kappa\C(a,b,c;a\oplus(b\oplus c)),
\]
and trace the image of the pair $(\alpha_1,\id^3_{(a\oplus b)\oplus c})$, which will turn out to be $\alpha_3$.  
The composition unpacks using the definition 
just as an ordinary composite, since the usual first step is just the identity, and suppressed, since we are assuming $\kappa_1=1$:
\[
\xymatrix{
\C((a\oplus b)\oplus c,a\oplus(b\oplus c))\times\C((a\oplus b)\oplus c,(a\oplus b)\oplus c)
\ar[d]^-{\circ}
\\\C((a\oplus b)\oplus c,a\oplus(b\oplus c)).
}
\]
Tracing the pair $(\alpha_1,\id^3_{(a\oplus b)\oplus c})$ through this, we just get $\alpha$, but now interpreted
as an element of $U_\kappa\C(a,b,c;(a\oplus b)\oplus c)$, so we actually end up with $\alpha_3$.  We can therefore
apply $\Fhat$ throughout, and find that the pair $(\Fhat(\alpha_1),\Fhat(\id^3_{(a\oplus b)\oplus c}))$ is sent to $\Fhat(\alpha_3)$.  But tracing
that through the definition, which just unpacks
\[
\xymatrix{
\makebox[.7\textwidth][c]{$U_\kappa\D(\Fhat((a\oplus b)\oplus c),\Fhat(a\oplus(b\oplus c)))\times U_\kappa\D(\Fhat a,\Fhat b,\Fhat c;\Fhat((a\oplus b)\oplus c))$}
\ar[d]^-{\Gamma}
\\U_\kappa\D(\Fhat a,\Fhat b,\Fhat c;\Fhat(a\oplus(b\oplus c)))
}
\]
as
\[
\xymatrix{
\makebox[.7\textwidth][c]{$\D(F((a\oplus b)\oplus c),F(a\oplus(b\oplus c)))\times\D((Fa\oplus Fb)\oplus Fc,F((a\oplus b)\oplus c))$}
\ar[d]^-{\circ}
\\\D((Fa\oplus Fb)\oplus Fc,F(a\oplus(b\oplus c))),
}
\]
we find that
\[
\Fhat(\alpha_3)=\Fhat(\alpha_1)\circ \Fhat(\id_{(a\oplus b)\oplus c})=F(\alpha)\circ\xi\circ(\xi\oplus1),
\]
where we are justified in removing the hat at the last step since the induced functor $F$ is just the multifunctor $\Fhat$ at the 1-level.
This is the counterclockwise direction of our desired diagram.

In the other direction, we begin with the composition in $U_\kappa\C$ given by
\[
U_\kappa\C(a,b\oplus c;a\oplus(b\oplus c))\times U_\kappa\C(a;a)\times U_\kappa\C(b,c;b\oplus c)
\to U_\kappa\C(a,b,c;a\oplus(b\oplus c)),
\]
and trace the triple given by $(\id^2_{a\oplus(b\oplus c)},\id^1_a,\id^2_{b\oplus c})$ through the expanded definition of the composition:
\[
\xymatrix{
\C(m\cdot(a,b\oplus c),a\oplus(b\oplus c))\times\C^2((a,b\oplus c),(a,b\oplus c))
\ar[d]^-{1\times m}
\\\C(m\cdot(a,b\oplus c),a\oplus(b\oplus c))\times\C(m\cdot(a,b\oplus c),m\cdot(a,b\oplus c))
\ar[d]^-{\circ}
\\\C(m\cdot(a,b\oplus c),a\oplus(b\oplus c))
\ar[d]^-{\C(\phi((a,b\oplus c),(a,b,c)),1)}
\\\C((a\oplus b)\oplus c,a\oplus(b\oplus c)).
}
\]
The triple ends up at $\id_{a\oplus(b\oplus c)}$ before the final map, but the map $\phi((a,b\oplus c),(a,b,c))$ is induced
by the map in $Y(3)$ 
\[
\kappa_3=\gamma(m;m,1)\to\gamma(m;1,m),
\]
which is precisely the associativity isomorphism.  We can therefore conclude that our triple gets sent to $\alpha$, 
interpreted as $\alpha_3$, so
applying the multifunctor $\Fhat$ throughout, we will end up at $\Fhat(\alpha_3)$.

When we do so, we start with the triple 
\[
(\Fhat(\id^2_{a\oplus (b\oplus c)}),\Fhat(\id^1_a),\Fhat(\id^2_{b\oplus c}))=(\xi_{a,b\oplus c},\id_{\Fhat a},\xi_{b,c})
\]
in the composition
\[
\xymatrix{
\makebox[.7\textwidth][c]{$U_\kappa\D(\Fhat a,\Fhat(b\oplus c);\Fhat(a\oplus(b\oplus c)))\times U_\kappa\D(\Fhat a;\Fhat a)\times U_\kappa\D(\Fhat b,\Fhat c;\Fhat(b\oplus c))$}
\ar[d]^-{\Gamma}
\\U_\kappa\D(\Fhat a,\Fhat b,\hat Fc;\Fhat(a\oplus(b\oplus c))).
}
\]
Expanding using the definition of $\Gamma$, we get
\[
\xymatrix{
\makebox[.7\textwidth][c]{$\D(m\cdot(Fa,F(b\oplus c)),F(a\oplus(b\oplus c)))\times\D^2((Fa,m\cdot(Fb,Fc)),(Fa,F(b\oplus c)))$}
\ar[d]^-{1\times m}
\\\makebox[.7\textwidth][c]{$\D(m\cdot(Fa,F(b\oplus c)),F(a\oplus(b\oplus c)))\times\D(m\cdot(Fa,m\cdot(Fb,Fc)),m\cdot(Fa,F(b\oplus c)))$}
\ar[d]^-{\circ}
\\\D(m\cdot(Fa,m\cdot(Fb,Fc)),F(a\oplus(b\oplus c)))
\ar[d]^-{\D(\phi((Fa,F(b\oplus c)),(Fa,Fb,Fc)),1)}
\\\D((Fa\oplus Fb)\oplus Fc,F(a\oplus(b\oplus c))).
}
\]
Now $\phi((Fa,F(b\oplus c)),(Fa,Fb,Fc))$ is again induced by the unique isomorphism in $Y(3)$
\[
\gamma(m;m,1)\to\gamma(m;1,m),
\]
which is precisely the isomorphism inducing the associativity isomorphism $\alpha$.  Tracing our triple $(\xi_{a,b\oplus c},\id_{Fa},\xi_{b,c})$
through the composition, we get
\[
(\xi_{a,b\oplus c},\id_{Fa},\xi_{b,c})
\mapsto
(\xi_{a,b\oplus c},\id_a\oplus \xi_{b,c})
\mapsto
\xi\circ(1\oplus\xi)
\mapsto
\xi\circ(1\oplus\xi)\circ\alpha.
\]
We conclude that in this direction, we have
\[
\Fhat(\alpha_3)=\xi\circ(1\oplus\xi)\circ\alpha,
\]
so identifying the two calculations of $\Fhat(\alpha_3)$, we find that
\[
F(\alpha)\circ\xi\circ(\xi\oplus1)=\Fhat(\alpha_3)=\xi\circ(1\oplus\xi)\circ\alpha,
\]
which says precisely that our desired coherence diagram for associativity does commute.
\end{proof}

We conclude this direction of the argument by verifying commutativity of the transposition coherence diagram,
which we rewrite by switching our two variables.

\begin{lemma}
The transposition coherence diagram
\[
\xymatrix{
Fb\oplus Fa
\ar[r]^-{\tau}
\ar[d]_-{\xi}
&Fa\oplus Fb
\ar[d]^-{\xi}
\\F(b\oplus a)
\ar[r]_-{F\tau}
&F(a\oplus b)
}
\]
commutes.
\end{lemma}

\begin{proof}
Once again we have a morphism in $\C$ that represents two different morphisms in $U_\kappa\C$: in this case 
the transposition isomorphism $\tau\in\C(b\oplus a,a\oplus b)$ represents a 1-morphism in $U_\kappa\C(b\oplus a;a\oplus b)$ which
we will write as $\tau_1$, and also a 2-morphism in $U_\kappa\C(b,a;a\oplus b)$ which we will write as $\tau_2$.

Next, since $\Fhat$ is a multifunctor, it is in particular equivariant, so the diagram
\[
\xymatrix{
U_\kappa\C(a,b;a\oplus b)
\ar[r]^-{\Fhat}
\ar[d]_-{\tau^*}
&U_\kappa\D(\Fhat a,\Fhat b;\Fhat(a\oplus b))
\ar[d]^-{\tau^*}
\\U_\kappa\C(b,a;a\oplus b)
\ar[r]_-{\Fhat}
&U_\kappa\D(\Fhat b,\Fhat a;\Fhat(a\oplus b))
}
\]
must commute.  Tracing the element $\id^2_{a\oplus b}$ clockwise in the underlying categories from the upper left of the diagram, we have
\[
\id_{a\oplus b}
\mapsto
\xi
\mapsto
\xi\circ\tau,
\]
as an element of $\D(Fb\oplus Fa,F(a\oplus b))$,
which is the clockwise composite in the diagram we wish to verify.

Now tracing $\id^2_{a\oplus b}$ counterclockwise, we get
\[
\id_{a\oplus b}
\mapsto
\tau_2
\mapsto
\Fhat(\tau_2).
\]
Since the diagram commutes, we conclude that
\[
\Fhat(\tau_2)=\xi\circ\tau.
\]

Next, we have the composition in $U_\kappa\C$
\[
\xymatrix{
U_\kappa\C(b\oplus a,a\oplus b)\times U_\kappa\C(b,a;b\oplus a)
\ar[r]^-{\Gamma}
&U_\kappa\C(b,a;a\oplus b),
}
\]
and expanding using the definition of composition in $U_\kappa\C$, we see that the pair $(\tau_1,\id^2_{b\oplus a})$ gets sent to
$\tau_2$.  Applying the multifunctor $\Fhat$, it follows that the pair $(\Fhat(\tau_1),\Fhat(\id^2_{b\oplus a}))$ gets sent to $\Fhat(\tau_2)$.  
But this then says that under
\[
\xymatrix{
U_\kappa\D(F(b\oplus a),F(a\oplus b))\times U_\kappa\D(Fb,Fa;F(b\oplus a))
\ar[r]^-{\Gamma}
&U_\kappa\D(Fb,Fa;F(a\oplus b)),
}
\]
the pair $(\Fhat(\tau_1),\xi)$ gets sent to $\Fhat(\tau_2)$.  Since $F$ is just the restriction of $\Fhat$ to 1-morphisms, we see that
$\Fhat(\tau_1)=F(\tau)$, and further the definition of $\Gamma$ in the display sends $(\Fhat(\tau_1),\xi)$ to $\Fhat(\tau_1)\circ\xi$.
Combining the two calculations, we conclude that
\[
\xi\circ\tau=F(\tau_2)=F(\tau)\circ\xi,
\]
which says that our desired diagram commutes.  We therefore do have a lax symmetric monoidal functor.
\end{proof}

We turn now to the reverse direction of the theorem: given a lax symmetric monoidal functor $F$, we must produce
a multifunctor $U_\kappa F:U_\kappa\C\to U_\kappa\D$ on underlying multicategories.  Again, since all underlying multicategories are canonically and coherently
isomorphic, it suffices to produce a multifunctor between the underlying multicategories given by the sequence
of objects $\kappa_n\in Y(n)$, where $\kappa_0=0$, $\kappa_1=1$, and for $n\ge2$, $\kappa_n=\gamma(m;\kappa_{n-1},1)$.
We begin by generalizing the structure map $\xi:Fa\oplus Fb\to F(a\oplus b)$.  Note that we can express $\xi$ as
\[
\xi:\kappabar(Fa,Fb)\to F\kappabar(a,b).
\]

\begin{definition}
Let $\bd x$ be an object of $\C^n$, with $F\bd x$ the corresponding object of $\D^n$, so
\[
F\bd x=(Fx_1,\dots,Fx_n).
\]
We define a map
\[
\xi_n:\kappabar F\bd x\to F\kappabar\bd x
\]
first for $n=0$, in which case $\bd x=*\in\C^0$, $\kappabar\bd x=e_\C$ and $\kappabar F\bd x=e_\D$.  Then we define $\xi_0$ to be
\[
\xymatrix{
\kappabar F\bd x=e_\D
\ar[r]^-{\eta}
&Fe_\C=F\kappabar\bd x.
}
\]
For $n\ge1$ we use induction on $n$, with $\xi_1=\id_{Fx_1}$.  For $n\ge2$, let $\xhat=(x_1,\dots,x_{n-1})$, that is, $\bd x$ with the last entry
deleted.  Note that since $\kappa_n=\gamma(m;\kappa_{n-1},1)$, we have
\[
\kappabar\bd x=\kappabar\xhat\oplus x_n.
\]
Then we define $\xi_n$ as the composite
\[
\xymatrix@C+15pt{
\kappabar F\bd x=\kappabar F\xhat\oplus Fx_n
\ar[r]^-{\xi_{n-1}\oplus1}
&F\kappabar\xhat\oplus Fx_n
\ar[r]^-{\xi}
&F(\kappabar\xhat\oplus x_n)=F\kappabar\bd x.
}
\]
In particular, $\xi_2=\xi$.  
\end{definition}

\begin{remark}
We can't actually start our induction at $n=0$, because the definition
\[
\kappa_n=\gamma(m;\kappa_{n-1},1)
\]
does not apply when $n=1$: we have $\kappa_1=1$, but $\gamma(m;0,1)\ne1$ in $Y(1)$.  This is reflected in the fact that
in a general symmetric monoidal category, we don't have $x=e_\C\oplus x$: they're canonically isomorphic, but not equal.
\end{remark}

This now allows us to define the structure of our underlying multifunctor $U_\kappa F$, which we will often just write as $F$.

\begin{definition}
Let $F:\C\to\D$ be a lax symmetric monoidal functor with structure maps $\eta:e_\D\to Fe_\C$ and $\xi:Fa\oplus Fb\to F(a\oplus b)$.
We define the underlying multifunctor by giving it on objects and morphism sets.  On objects, we just use the map given
by the functor $F$.  For morphism sets, let $\bd x\in\C^n$ and $y\in\C$; we must produce a map
\[
F_n:U_\kappa\C(\bd x;y)\to U_\kappa\D(F\bd x;Fy).
\]
For all $n$, we define $F_n$ as the composite
\[
\xymatrix@C+5pt{
U_\kappa\C(\bd x;y)=\C(\kappabar\bd x,y)
\ar[r]^-{F}
&\D(F\kappabar\bd x;Fy)
\ar[r]^-{\D(\xi_n,1)}
&\D(\kappabar F\bd x;Fy)=U_\kappa\D(F\bd x;Fy).
}
\]
Note that when $n=0$, this becomes
\[
\xymatrix{
U_\kappa\C(;y)=\C(e_\C,y)
\ar[r]^-{F}
&\D(Fe_\C,Fy)
\ar[r]^-{\D(\eta,1)}
&\D(e_\D,Fy)=U_\kappa\D(;Fy).
}
\]
This completes the definition of the structure of $U_\kappa F$.
Note that in the special case $n=1$, since $\kappabar x=x$ (because $\kappa_1=1$), $F_1$ coincides with
the original functor $F$.
\end{definition}

We must show that $U_\kappa F$ preserves all the multicategory structure.  In particular, it must preserve
\begin{enumerate}
\item
the identity maps,
\item
the composition, and
\item
the symmetric group actions.
\end{enumerate}

For preservation of the identity maps, we merely note that $F_1$ coincides with the original functor, as noted above,
and therefore the identity maps are preserved.

We turn to verifying the preservation of composition, which requires some preliminary definitions and lemmas.  

\begin{definition}
Let $\bd x\in\C^j$, let $j=q+r$, and let $\bd x=\bd x'\odot\bd x''$ with $\bd x'\in\C^q$ and $\bd x''\in \C^r$, so
$\bd x'=(x_1,\dots,x_q)$ and $\bd x''=(x_{q+1},\dots,x_j)$.  We allow the possibility that either $q=0$ or $r=0$,
but do assume that $j>0$.
Then we
define a map
\[
\phi_{qr}\bd x:\kappabar\bd x\to\kappabar\bd x'\oplus\kappabar\bd x''
\]
by applying the unique isomorphism
\[
\phi_{qr}:\kappa_j\to\gamma(m;\kappa_q,\kappa_r)
\]
in $Y(j)$ to the object $\bd x$.  
\end{definition}

Our first lemma towards preservation of composition relates these maps $\phi_{qr}$ to the previously defined $\xi_j$'s:

\begin{lemma}\label{xi1}
The following diagram commutes:
\[
\xymatrix@C+15pt{
\kappabar F\bd x
\ar[r]^-{\phi_{qr}F\bd x}
\ar[dd]_-{\xi_j}
&\kappabar F\bd x'\oplus\kappabar F\bd x''
\ar[d]^-{\xi_q\oplus\xi_r}
\\&F\kappabar\bd x'\oplus F\kappabar\bd x''
\ar[d]^-{\xi}
\\F\kappabar\bd x
\ar[r]_-{F\phi_{qr}\bd x}
&F(\kappabar\bd x'\oplus\kappabar\bd x'').
}
\]
\end{lemma}

\begin{proof}
We begin with the special case $r=0$, which requires its own argument.  Applying the map in $Y(j)$
\[
\phi_{j0}:\kappa_j\to\gamma(m;\kappa_j,0)
\]
to $\bd x\in\C^j$ gives us the map
\[
\phi_{j0}\bd x:\kappabar\bd x\to\kappabar\bd x\oplus e_\C
\]
inverse to the unit map $c:\kappabar\bd x\oplus e_\C\to\kappabar\bd x$.
Now our diagram becomes
\[
\xymatrix@C+15pt{
\kappabar F\bd x
\ar[r]^-{\phi_{j0}F\bd x}
\ar[dd]_-{\xi_j}
&\kappabar F\bd x\oplus e_\D
\ar[d]^-{\xi_j\oplus\xi_0}
\\&F\kappabar\bd x\oplus Fe_\C
\ar[d]^-{\xi}
\\F\kappabar\bd x
\ar[r]_-{F\phi_{j0}\bd x}
&F(\kappabar\bd x\oplus e_\C).
}
\]
But we can expand this as follows:
\[
\xymatrix@C+15pt{
\kappabar F\bd x
\ar[r]^-{\phi_{j0}F\bd x}
\ar[dd]_-{\xi_j}
&\kappabar F\bd x\oplus e_\D
\ar[d]_-{\xi_j\oplus1}
\ar[dr]^-{\xi_j\oplus\xi_0}
\\&F\kappabar\bd x\oplus e_\D
\ar[r]_-{1\oplus\xi_0}
&F\kappabar\bd x\oplus Fe_\C
\ar[d]^-{\xi}
\\F\kappabar\bd x
\ar[ur]_-{c^{-1}}
\ar[rr]_-{F\phi_{j0}}
&&F(\kappabar\bd x\oplus e_\C).
}
\]
Remembering that $\phi_{j0}$ coincides with $c^{-1}$, the inverse of the unit isomorphism, the left square commutes
by naturality of $c^{-1}$, and the bottom square by the coherence diagram relating $\eta$ and $\xi$.  The triangle
commutes by inspection, and now the perimeter gives us our desired diagram.  

We also need to consider the special case in which $q=0$ and $r=1$.  In this case, we have the map
\[
\phi_{01}:1=\kappa_1\to\gamma(m;\kappa_0,\kappa_1)=\gamma(m;0,1),
\]
which applies to an object $x\in\C$ to give $x\to e_\C\oplus x$, the inverse of the composite
\[
\xymatrix{
e_\C\oplus x
\ar[r]^-{\tau}
&x\oplus e_\C
\ar[r]^-{c}
&x,
}
\]
since $\gamma(m;0,1)\cdot\tau=\gamma(m\cdot\tau;1,0)$.  Consequently we can write
\[
\phi_{01}x=\tau c^{-1}:x\to e_\C\oplus x.
\]
Further, $\tau c^{-1}$ (or equivalently $c\tau$) satisfies the analogous coherence diagram with $\xi$ that $c$
itself does, namely
\[
\xymatrix{
e_\D\oplus Fx
\ar[r]^-{\eta\oplus1}
\ar[d]_-{c\tau}
&Fe_\C\oplus Fx
\ar[d]^-{\xi}
\\Fx
&F(e_\C\oplus x).
\ar[l]^-{F(c\tau)}
}
\]
This is because we can expand it as follows:
\[
\xymatrix{
e_\D\oplus Fx
\ar[rr]^-{\eta\oplus1}
\ar[d]_-{\tau}
&&Fe_\C\oplus Fx
\ar[dl]_-{\tau}
\ar[d]^-{\xi}
\\Fx\oplus e_\D
\ar[d]_-{c}
\ar[r]^-{1\oplus\eta}
&Fx\oplus Fe_\C
\ar[dr]_-{\xi}
&F(e_\C\oplus x)
\ar[d]^-{F\tau}
\\Fx
&&F(x\oplus e_\C).
\ar[ll]^-{Fc}
}
\]
The top square commutes by naturality of $\tau$, the right ``square'' by the coherence diagram for $\xi$ and $\tau$, and the bottom
square by the coherence diagram for $\eta$ and $\xi$.  The perimeter traces the claimed diagram.

Now the desired diagram in the case $q=0$ and $r=1$ becomes
\[
\xymatrix@C+15pt{
Fx
\ar[r]^-{\tau c^{-1}}
\ar[dd]_-{=}
&e_\D\oplus Fx
\ar[d]^-{\xi_0\oplus\xi_1}
\\&Fe_\C\oplus Fx
\ar[d]^-{\xi}
\\Fx
\ar[r]_-{F(\tau c^{-1})}
&F(e_\C\oplus x).
}
\]
But since $\xi_0=\eta$ and $\xi_1=\id$, this is just a rearrangement of the previous coherence diagram, so it does commute.  
This gives us the special case $q=0$ and $r=1$.

We now proceed by induction on $r$ starting at $r=1$.
If $q=0$, we have just verified the lemma, and if $q\ge1$, then
$\phi_{qr}=\id$, and the diagram simply gives the definition
of $\xi_n$.  

Now assume by induction that $r>1$ and that the diagram commutes with $r$ replaced with $r-1$.  We examine two diagrams that will
need to be pasted together horizontally, and hats always indicate that the last entry is deleted.  The first diagram is
\[
\xymatrix@C+30pt{
\kappabar F\bd x
\ar[r]^-{=}
\ar[ddd]_-{\xi_j}
&\kappabar F\xhat\oplus Fx_j
\ar[r]^-{\phi_{q(r-1)}\oplus 1}
\ar[dd]^-{\xi_{j-1}\oplus1}
&(\kappabar F\bd x'\oplus\kappabar F\bd\xhat'')\oplus Fx_j
\ar[d]^-{(\xi_q\oplus\xi_{r-1})\oplus1}
\\&&(F\kappabar\bd x'\oplus F\kappabar\bd\xhat'')\oplus Fx_j
\ar[d]^-{\xi\oplus1}
\\&F\kappabar\xhat\oplus Fx_j
\ar[r]^-{F\phi_{q(r-1)}\oplus1}
\ar[d]^-{\xi}
&F(\kappabar\bd x'\oplus\kappabar\xhat'')\oplus Fx_j
\ar[d]^-{\xi}
\\F\kappabar\bd x
\ar[r]_-{=}
\ar[d]_-{F\phi_{qr}}
&F(\kappabar\xhat\oplus x_j)
\ar[r]_-{F(\phi_{q(r-1)}\oplus1)}
&F((\kappabar\bd x'\oplus\kappabar\xhat'')\oplus x_j)
\ar[d]^-{F\alpha}
\\F(\kappabar\bd x'\oplus\kappabar\bd x'')
\ar[rr]_-{=}
&&F(\kappabar\bd x'\oplus(\kappabar\xhat''\oplus x_j)).
}
\]
The top left rectangle commutes by the definition of $\xi_j$, the top right rectangle by induction, the middle right
rectangle by naturality of $\xi$, and the bottom rectangle because it is $F$ applied to a diagram induced by a diagram in $Y(j)$,
where all diagrams commute.  The whole diagram therefore commutes.

The second diagram is
\[
\xymatrix{
(\kappabar F\bd x'\oplus\kappabar F\xhat'')\oplus Fx_j
\ar[r]^-{\alpha}
\ar[d]_-{(\xi_{q}\oplus\xi_{r-1})\oplus1}
&\kappabar F\bd x'\oplus(\kappabar F\xhat''\oplus Fx_j)
\ar[r]^-{=}
\ar[d]^-{\xi_q\oplus(\xi_{r-1}\oplus1)}
&\kappabar F\bd x'\oplus\kappabar F\bd x''
\ar[dd]^-{\xi_q\oplus\xi_r}
\\ (F\kappabar\bd x'\oplus F\kappabar\xhat'')\oplus Fx_j
\ar[r]_-{\alpha}
\ar[d]_-{\xi\oplus1}
&F\kappabar\bd x'\oplus(F\kappabar\xhat\oplus Fx_j)
\ar[d]^-{1\oplus\xi}
\\F(\kappabar\bd x'\oplus\kappabar\xhat'')\oplus Fx_j
\ar[d]_-{\xi}
&F\kappabar\bd x'\oplus F(\kappabar\xhat''\oplus x_j)
\ar[r]^-{=}
\ar[d]^-{\xi}
&F\kappabar\bd x'\oplus F\kappabar\bd x''
\ar[d]^-{\xi}
\\F((\kappabar\bd x'\oplus\kappabar\xhat'')\oplus x_j)
\ar[r]^-{F\alpha}
\ar[d]_-{F\alpha}
&F(\kappabar\bd x'\oplus(\kappabar\xhat''\oplus x_j))
\ar[r]^-{=}
&F(\kappabar\bd x'\oplus\kappabar\bd x'')
\ar[d]^-{=}
\\F(\kappabar\bd x'\oplus(\kappabar\xhat''\oplus x_j))
\ar[rr]_-{=}
&&F(\kappabar\bd x'\oplus\kappabar\bd x'').
}
\]
The top left square commutes by naturality of $\alpha$, the top right by definition of $\xi_r$, the middle left by the
coherence diagram for $\alpha$, and the middle right and bottom by inspection.  The total second diagram therefore
also commutes.

The right column of the first diagram coincides with the left column of the second diagram, so we can paste the two
diagrams together along their common column.  When we do so, the counterclockwise direction of the total diagram
coincides with the counterclockwise direction of our desired diagram, but the clockwise direction requires a bit more 
work.  We need the following diagram to commute:
\[
\xymatrix@C+20pt{
\kappabar F\bd x
\ar[r]^-{=}
\ar[d]_-{\phi_{qr}}
&\kappabar F\xhat\oplus Fx_j
\ar[r]^-{\phi_{q(r-1)}\oplus1}
&(\kappabar F\bd x'\oplus\kappabar F\xhat'')\oplus Fx_j
\ar[d]^-{\alpha}
\\\kappabar F\bd x'\oplus\kappabar F\bd x''
\ar[rr]_-{=}
&&\kappabar F\bd x'\oplus(\kappabar F\xhat''\oplus Fx_j).
}
\]
However, all three maps are induced by maps in $Y(j)$, where all diagrams commute.  The desired diagram therefore
does commute.
\end{proof}

We will need a generalization of this lemma, and this requires a bit more notation.

\begin{definition}
Let $\bd x\in\C^j$, and let $j=j_1+\cdots+j_n$, so we can decompose $\bd x$ as
\[
\bd x=\odot_s\bd x_s,
\]
where $\bd x_s\in\C^{j_s}$ for $1\le s\le n$.  Regardless of $\bd x$, we have a unique isomorphism
\[
\kappa_j\to\gamma(\kappa_n;\br{\kappa_{j_s}})
\]
in $Y(j)$.  
Note that if we apply the target element $\gamma(\kappa_n;\br{\kappa_{j_s}})$ to $\bd x=\odot_s\bd x_s$, we can write the result as either
$\kappabar\br{\kappabar\bd x_s}$ or $\kappa_n\br{\kappabar\bd x_s}$, since we have chosen a single object $\kappa_n\in Y(n)$
as the target of the $n$'th map in our defining sequence $\{\kappa_n\}$.  
We will use $\kappa_n\br{\kappabar\bd x_s}$ since we need the index $n$ for induction in the next proof.
So applying the map in $Y(n)$ to
$\bd x=\odot_s\bd x_s$, we obtain a map (which is an isomorphism) which we denote
\[
\phi\br{j_s}:\kappabar(\odot_s \bd x_s)\to\kappa_n\br{\kappabar \bd x_s}.
\]
This is a generalization of the previous $\phi_{qr}$ when $n=2$, $j_1=q$, and $j_2=r$.
\end{definition}

We also generalize lemma \ref{xi1} for these maps, as follows.

\begin{lemma}\label{xi2}
Given $\bd x\in\C^j$ and $j=j_1+\cdots+j_n$, 
note that $\odot_sF\bd x_s=F(\odot_s\bd x_s)$.  Then
decomposing $\bd x$ as $\odot_s\bd x_s$ for $1\le s\le n$, 
the following diagram commutes:
\[
\xymatrix{
\kappabar(\odot_s F\bd x_s)
\ar[r]^-{\phi\br{j_s}}
\ar[dd]_-{\xi_j}
&\kappa_n\br{\kappabar F\bd x_s}
\ar[d]^-{\kappa_n\br{\xi_{j_s}}}
\\&\kappa_n\br{F\kappabar\bd x_s}
\ar[d]^-{\xi_n}
\\F\kappabar(\odot_s\bd x_s)
\ar[r]_-{F\phi\br{j_s}}
&F\kappa_n\br{\kappabar\bd x_s}.
}
\]
\end{lemma}

\begin{proof}
We proceed by induction on $n$, and the case $n=1$ is trivial: $\phi\br{j}=\id$, $\kappa_1=1$, and $\xi_1=\id$ as well.

As with Lemma \ref{xi1}, the proof consists of verifying two diagrams that need to be pasted together
horizontally.  We let $j'=j-j_n$, and let hats throughout indicate that the last index has been deleted.  
The notation $\phi_{qr}$ is as above from Lemma \ref{xi1}.
Our first
diagram is as follows:
\[
\xymatrix@C+15pt{
\kappabar(\odot_sF\bd x_s)
\ar[r]^-{\phi_{j'j_n}}
\ar[ddd]_-{\xi_j}
&\kappabar(\hat{\odot}_sF\bd x_s)\oplus\kappabar F\bd x_n
\ar[r]^-{\phi\widehat{\br{j_s}}\oplus1}
\ar[dd]^-{\xi_{j'}\oplus\xi_{j_n}}
&\kappa_{n-1}\widehat{\br{\kappabar F\bd x_s}}\oplus\kappabar F\bd x_n
\ar[d]^-{\kappa_{n-1}\widehat{\br{\xi_{j_s}}}\oplus1}
\\&&\kappa_{n-1}\widehat{\br{F\kappabar\bd x_s}}\oplus \kappabar F\bd x_n
\ar[d]^-{\xi_{n-1}\oplus\xi_{j_n}}
\\&F\kappabar(\hat\odot_s\bd x_s)\oplus F\kappabar\bd x_n
\ar[d]^-{\xi}
\ar[r]^-{F\phi\widehat{\br{j_s}}\oplus1}
&F\kappa_{n-1}\widehat{\br{\kappabar\bd x_s}}\oplus F\kappabar\bd x_n
\ar[d]^-{\xi}
\\F\kappabar(\odot_s\bd b_s)
\ar[d]_-{F\phi\br{j_s}}
\ar[r]^-{F\phi_{j'j_n}}
&F(\kappabar(\hat\odot_s\bd x_s)\oplus\kappabar\bd x_n)
\ar[r]^-{F(\phi\widehat{\br{j_s}}\oplus1)}
&F(\kappa_{n-1}\widehat{\br{\kappabar\bd x_s}}\oplus\kappabar\bd x_n).
\\F\kappa_n\br{\kappabar\bd x_s}
\ar[urr]_-{=}
}
\]
The left rectangle commutes since it is an instance of Lemma \ref{xi1}.  The top right rectangle commutes by induction,
and the bottom right rectangle by naturality of $\xi$.  The bottom triangle commutes since it is $F$ applied to a diagram
induced from maps in $Y(j)$, where all diagrams commute.  The total diagram therefore commutes.

The second diagram is as follows:
\[
\xymatrix{
\kappa_{n-1}\widehat{\br{\kappabar F\bd x_s}}\oplus\kappabar F\bd x_n
\ar[rr]^-{=}
\ar[d]_-{\kappa_{n-1}\widehat{\br{\xi_{j_s}}}\oplus1}
\ar[dr]^-{\kappa_{n-1}\widehat{\br{\xi_{j_s}}}\oplus\xi_{j_n}}
&&\kappa_n\br{\kappabar F\bd x_s}
\ar[d]^-{\kappa_n\br{\xi_{j_s}}}
\\\kappa_{n-1}\widehat{\br{F\kappabar\bd x_s}}\oplus\kappabar F\bd x_n
\ar[r]^-{1\oplus\xi_{j_n}}
\ar[d]_-{\xi_{n-1}\oplus\xi_{j_n}}
&\kappa_{n-1}\widehat{\br{F\kappabar\bd x_s}}\oplus F\kappabar\bd x_n
\ar[r]^-{=}
\ar[dl]^-{\xi_{n-1}\oplus1}
&\kappa_n\br{F\kappabar\bd x_s}
\ar[dd]^-{\xi_n}
\\F\kappa_{n-1}\widehat{\br{\kappabar\bd x_s}}\oplus F\kappabar\bd x_n
\ar[d]_-{\xi}
\\F(\kappa_{n-1}\widehat{\br{\kappabar\bd x_s}}\oplus \kappabar\bd x_n)
\ar[rr]_-{=}
&&F\kappa_n\br{\kappabar\bd x_s}.
}
\]
The two triangles commute by inspection.  The top (distorted) rectangle commutes since $\kappa_n\bd y=\kappa_{n-1}\hat{\bd y}\oplus y_n$,
where in this case $\bd y=\br{\kappabar F\bd x_s}$,
and the bottom part commutes by definition of $\xi_n$.  The total diagram therefore commutes.

Now we paste the two diagrams together along their common column, the right column of the first diagram and the left column of
the second diagram, and see that the 
counterclockwise direction of the total diagram coincides with the counterclockwise direction of our desired diagram.  
For the clockwise direction, again we need to do a bit more work.  We
observe that since the top row of the total diagram is induced by maps in $Y(n)$, where all diagrams commute, it does
coincide with the desired clockwise direction, and the desired diagram therefore does commute.  
\end{proof}

We are now ready to verify that the multifunctor $F$ preserves composition.  
Let $\bd x_s\in\C^{j_s}$ for $1\le s\le n$, $\bd y\in\C^n$, and $z\in\C$.
We claim:

\begin{proposition}
The multifunctor $F$ preserves composition, meaning the diagram
\[
\xymatrix{
U_\kappa\C(\bd y;z)\times\prod_{s=1}^n U_\kappa\C(\bd x_s,y_s)
\ar[d]_-{\Gamma}
\ar[r]^-{F}
&U_\kappa\D(F\bd y,Fz)\times\prod_{s=1}^n U_\kappa\D(F\bd x_s,Fy_s)
\ar[d]^-{\Gamma}
\\U_\kappa\C(\odot_s\bd x_s,z)
\ar[r]_-{F}
&U_\kappa\D(\odot_s F\bd x_s,Fz)
}
\]
commutes.
\end{proposition}

\begin{proof}
We proceed by a sequence of three diagrams which can be pasted together horizontally, resulting in a total diagram
that gives our desired diagram, unpacked using the definitions of our multicategory composition $\Gamma$ and the multifunctor $F$.
Note that $\kappa_n\bd y=\kappabar\bd y$ because of our convention that $\kappa_n$ is a constant function.
The first diagram is as follows:
\[
\xymatrix{
&\C(\kappabar\bd y,z)\times\C^n(\br{\kappabar\bd x_s},\bd y)
\ar[d]^-{1\times\kappa_n}
\\&\C(\kappabar\bd y,z)\times\C(\kappa_n\br{\kappabar\bd x_s},\kappabar\bd y)
\ar[dl]_-{\circ}
\ar[d]^-{F}
\\\C(\kappa_n\br{\kappabar\bd x_s},z)
\ar[dr]_-{F}
\ar[dd]_-{\C(\phi\br{j_s},1)}
&\D(F\kappabar\bd y,Fz)\times\D(F\kappa_n\br{\kappabar\bd x_s},F\kappabar\bd y)
\ar[d]^-{\circ}
\\&\D(F\kappa_n\br{\kappabar\bd x_s},Fz)
\ar[d]^-{\D(F\phi\br{j_s},1)}
\\\C(\kappabar(\odot_s\bd x_s),z)
\ar[r]_-{F}
&\D(F\kappabar(\odot_s\bd x_s),Fz).
}
\]
Both parts of the diagram commute because $F$ is a functor.

The second diagram, which pastes to the right of the first one, is as follows:
\[
\xymatrix@C+15pt{
\C(\kappabar\bd y,z)\times\C^n(\br{\kappabar\bd x_s},\bd y)
\ar[d]_-{1\times\kappa_n}
\ar[r]^-{F}
&\D(F\kappabar\bd y,Fz)\times\D^n(\br{F\kappabar\bd x_s},F\bd y)
\ar[d]^-{1\times\kappa_n}
\\\C(\kappabar\bd y,z)\times\C(\kappa_n\br{\kappabar\bd x_s},\kappabar\bd y)
\ar[d]_-{F}
&\D(F\kappabar\bd y,Fz)\times\D(\kappa_n\br{F\kappabar\bd x_s},\kappabar F\bd y)
\ar[d]^-{1\times\D(1,\xi_n)}
\\\D(F\kappabar\bd y,Fz)\times\D(F\kappa_n\br{\kappabar\bd x_s},F\kappabar\bd y)
\ar[d]_-{\circ}
\ar[r]^-{1\times\D(\xi_n,1)}
&\D(F\kappabar\bd y,Fz)\times\D(\kappa_n\br{F\kappabar\bd x_s},F\kappabar\bd y)
\ar[d]^-{\circ}
\\\D(F\kappa_n\br{\kappabar\bd x_s},Fz)
\ar[dd]_-{\D(F\phi\br{j_s},1)}
\ar[r]_-{\D(\xi_n,1)}
&\D(\kappa_n\br{F\kappabar\bd x_s},Fz)
\ar[d]^-{\D(\kappa_n\br{\xi_{j_s}},1)}
\\&\D(\kappa_n\br{\kappabar F\bd x_s},Fz)
\ar[d]^-{\D(\phi\br{j_s},1)}
\\\D(F\kappabar(\odot_s\bd x_s),Fz)
\ar[r]_-{\D(\xi_j,1)}
&\D(\kappabar(\odot_sF\bd x_s),Fz).
}
\]
The top hexagon 
is the product of two separate hexagons, and the first factor
commutes, since both directions are just $F$.  
For the second factor, we trace
a typical element 
$\br{f_s}\in\C^n(\br{\kappabar\bd x_s},\bd y)$ through the hexagon, and find that it commutes precisely when
\[
F\kappa_n\br{f_s}\circ\xi_n=\xi_n\circ\kappa_n\br{Ff_s},
\]
which is to say that the following diagram commutes:
\[
\xymatrix@C+10pt{
\kappa_n\br{F\kappabar\bd x_s}
\ar[r]^-{\kappa_n\br{Ff_s}}
\ar[d]_-{\xi_n}
&\kappa_nF\bd y
\ar[d]^-{\xi_n}
\\F\kappa_n\br{\kappabar\bd x_s}
\ar[r]_-{F\kappa_n\br{f_s}}
&F\kappabar\bd y.
}
\]
But this is just naturality of $\xi_n$, so the top hexagon does commute.

The middle rectangle commutes by inspection.

The bottom part of the diagram is $\D(\underline{\phantom m},1)$ applied to the diagram
\[
\xymatrix{
\kappabar(\odot_sF\bd x_s)
\ar[r]^-{\xi_j}
\ar[d]_-{\phi\br{j_s}}
&F\kappabar(\odot_s\bd x_s)
\ar[dd]^-{F\phi\br{j_s}}
\\\kappa_n\br{\kappabar F\bd x_s}
\ar[d]_-{\kappa_n\br{j_s}}
\\\kappa_n\br{F\kappabar\bd x_s}
\ar[r]_-{\xi_n}
&F\kappa_n\br{\kappabar\bd x_s},
}
\]
which is an instance of Lemma \ref{xi2}.  The second total diagram therefore commutes.

The third diagram we need to paste onto the second one is as follows:
\[
\xymatrix@C-85pt{
\D(F\kappabar\bd y,Fz)\times\D^n(\br{F\kappabar\bd x_s},F\bd y)
\ar[d]_-{1\times\kappa_n}
\ar[rr]^-{\D(\xi_n,1)\times\D^n(\br{\xi_{j_s}},1)}
&&\D(\kappabar F\bd y,Fz)\times\D^n(\br{\kappabar F\bd x_s},F\bd y)
\ar[d]^-{1\times\kappa_n}
\\\D(F\kappabar\bd y,Fz)\times\D(\kappa_n\br{F\kappabar\bd x_s},\kappabar F\bd y)
\ar[dd]_-{1\times\D(1,\xi_n)}
\ar[dr]^-{\D(\xi_n,1)\times1}
&&\D(\kappabar F\bd y,Fz)\times\D(\kappa_n\br{\kappabar F\bd x_s},\kappabar F\bd y)
\ar[dd]^-{\circ}
\\&\D(\kappabar F\bd y, Fz)\times\D(\kappa_n\br{F\kappabar\bd x_s},\kappabar F\bd y)
\ar[dd]^-{\circ}
\ar[ur]^-{1\times\D(\kappa_n\br{\xi_{j_s}},1)\,\,\,\,\,\,\,}
\\\D(F\kappabar\bd y,Fz)\times\D(\kappa_n\br{F\kappabar\bd x_s},F\kappabar\bd y)
\ar[dr]_-{\circ}
&&\D(\kappa_n\br{\kappabar F\bd x_s},Fz)
\ar[dd]^-{\D(\phi\br{j_s},1)}
\\&\D(\kappa_n\br{F\kappabar\bd x_s},Fz)
\ar[ur]_-{\,\,\,\,\,\,\D(\kappa_n\br{\xi_{j_s}},1)}
\\&&\D(\kappabar(\odot_sF\bd x_s),Fz).
}
\]
The counterclockwise direction traces the right hand column of the second diagram, so this diagram can be
pasted to it.  The top pentagon commutes by naturality of $\kappa_n$, the left square by naturality of $\xi_n$,
and the right square by inspection.  Now examining the perimeter of the total pasted diagram, we see that the
counterclockwise direction gives us the expansion of $F\circ\Gamma$ in our original claimed diagram,
and the clockwise direction gives us the expansion of $\Gamma\circ F$.  The diagram therefore commutes,
and $F$ preserves composition.
\end{proof}

Our final step in showing that $F$ is a multifunctor is preservation of the $\Sigma_n$-actions.  This is
the content of the following proposition:

\begin{proposition}
Let $\bd x\in\C^n$, $y\in\C$, and $\sigma\in\Sigma_n$.  Then the following diagram commutes:
\[
\xymatrix{
U_\kappa\C(\bd x;y)
\ar[r]^-{F}
\ar[d]_-{\sigma^*}
&U_\kappa\D(F\bd x;Fy)
\ar[d]^-{\sigma^*}
\\U_\kappa\C(\sigma^{-1}\bd x;y)
\ar[r]_-{F}
&U_\kappa\D(\sigma^{-1}F\bd x;Fy).
}
\]
\end{proposition}

\begin{proof}
The map $\sigma^*$ is induced by the map in $Y(n)$
\[
\theta(\bd x,\sigma):\kappa_n(\sigma^{-1}\bd x)\to\kappa_n\bd x\cdot\sigma,
\]
but since we are assuming that the map $\kappa_n:\Ob\C^n\to\Ob Y(n)$ is constant,
we will just abbreviate
this to
\[
\theta:\kappa_n\to\kappa_n\cdot\sigma.
\]
Now expanding the desired diagram using the definitions, we find that we wish to verify commutativity of
the following diagram:
\[
\xymatrix{
\C(\kappabar\bd x,y)
\ar[r]^-{F}
\ar[d]_-{=}
&\D(F\kappabar\bd x,Fy)
\ar[r]^-{\D(\xi_n,1)}
\ar[d]^-{=}
&\D(\kappabar F\bd x,Fy)
\ar[d]^-{=}
\\\C((\kappa_n\cdot\sigma)(\sigma^{-1}\bd x),y)
\ar[r]^-{F}
\ar[d]_-{\C(\theta,1)}
&\D(F(\kappa_n\cdot\sigma)(\sigma^{-1}\bd x),Fy)
\ar[r]^-{\D(\xi_n,1)}
\ar[d]^-{\D(F\theta,1)}
&\D((\kappa_n\cdot\sigma)(\sigma^{-1}F\bd x),Fy)
\ar[d]^-{\D(\theta,1)}
\\\C(\kappa_n\sigma^{-1}\bd x,y)
\ar[r]_-{F}
&\D(F\kappa_n\sigma^{-1}\bd x,Fy)
\ar[r]_-{D(\xi_n,1)}
&\D(\kappa_n\sigma^{-1}F\bd x,Fy).
}
\]
The top two squares commute by inspection, and the lower left one by functoriality of $F$.  This reduces the
argument to verifying the bottom right square, which is in turn induced from a square in $\D$
as follows:
\[
\xymatrix{
\kappa_n\sigma^{-1}F\bd x
\ar[r]^-{\xi_n}
\ar[d]_-{\theta}
&F\kappa_n\sigma^{-1}\bd x
\ar[d]^-{F\theta}
\\\kappa_n\cdot\sigma\sigma^{-1}F\bd x
\ar[d]_-{=}
&F\kappa_n\cdot\sigma\sigma^{-1}\bd x
\ar[d]^-{=}
\\\kappa_nF\bd x
\ar[r]_-{\xi_n}
&F\kappa_n\bd x.
}
\]

We need to show that this diagram commutes, and we proceed by induction on $n$, starting with $n=1$, in which case there
are no non-trivial permutations, so the diagram does commute.  

Next, we assume the diagram does commute with $n$ replaced by $n-1$, and first consider the case in which $\sigma=\hat\sigma\oplus1$
for $\hat\sigma\in\Sigma_{n-1}$.  
We write
\[
\hat\theta:\kappa_{n-1}\to\kappa_{n-1}\cdot\hat\sigma
\]
for the unique map in $Y(n-1)$.
Then the diagram
\[
\xymatrix{
\kappa_n=\gamma(m;\kappa_{n-1},1)
\ar[r]^-{\theta}
\ar[dr]_-{\gamma(m;\hat\theta,1)}
&\gamma(m;\kappa_{n-1},1)\cdot\sigma
\ar[d]^-{=}
\\&\gamma(m;\kappa_{n-1}\cdot\hat\sigma,1)
}
\]
in $Y(n)$ commutes, since all diagrams commute in $Y(n)$.  So if we apply this diagram to an object $\bd x$, we have
the diagram
\[
\xymatrix{
\kappa_n\bd x=\kappa_{n-1}\xhat\oplus x_n
\ar[r]^-{\theta}
\ar[dr]_-{\hat\theta\oplus1}
&\kappa_n\cdot\sigma\bd x
=\kappa_n(\hat\sigma\xhat\oplus x_n)
\ar[d]^-{=}
\\&\kappa_{n-1}\hat\sigma\xhat\oplus x_n,
}
\]
which tells us that in this case, we have $\theta=\hat\theta\oplus1$.

Now we can expand our desired diagram as follows:
\[
\xymatrix@C-5pt{
\kappa_n\sigma^{-1}F\bd x=\kappa_{n-1}\hat\sigma^{-1}F\xhat\oplus Fx_n
\ar[r]^-{\xi_{n-1}\oplus1}
\ar[d]_-{\theta=\hat\theta\oplus1}
&F\kappa_{n-1}\hat\sigma^{-1}\xhat\oplus Fx_n
\ar[r]^-{\xi}
\ar[d]^-{F\hat\theta\oplus1}
&F(\kappa_{n-1}\hat\sigma^{-1}\xhat\oplus x_n)
\ar[d]^-{F(\hat\theta\oplus1)}
\\\kappa_n\sigma\sigma^{-1}F\bd x=\kappa_{n-1}\hat\sigma\hat\sigma^{-1}F\xhat\oplus Fx_n
\ar[r]^-{\xi_{n-1}\oplus1}
\ar[d]_-{=}
&F\kappa_{n-1}\hat\sigma\hat\sigma^{-1}\xhat\oplus Fx_n
\ar[r]^-{\xi}
\ar[d]^-{=}
&F(\kappa_{n-1}\hat\sigma\hat\sigma^{-1}\xhat\oplus x_n)
\ar[d]^-{=}
\\\kappa_nF\bd x=\kappa_{n-1}F\xhat\oplus Fx_n
\ar[r]_-{\xi_{n-1}\oplus1}
&F\kappa_{n-1}\xhat\oplus Fx_n
\ar[r]_-{\xi}
&F\kappa_n\bd x.
}
\]
Since $\xi_n=\xi\circ(\xi_{n-1}\oplus1)$ and $\theta=\hat\theta\oplus1$, the perimeter of the diagram does give our desired diagram.
The left part of the diagram commutes by induction, and the right part by naturality of $\xi$.  This concludes the verification of
preservation of the action of $\sigma$ when $\sigma=\hat\sigma\oplus1$.

The remaining case is when $\sigma$ moves the index $n$, but since all elements of $\Sigma_n$ can be written as a product
of transpositions of adjacent indices, we can restrict our attention to such transpositions, and the only one that moves the last index
is the transposition of $n$ and $n-1$.  So we let $\sigma$ be this transposition, and the argument will
be done when we've verified the square for this particular permutation, which satisfies $\sigma=\sigma^{-1}$.  

We first look at the inducing map $\theta:\kappa_n\to\kappa_n\cdot\sigma$, and note that applied to an object $\bd x\in\C^n$, we
have
\[
\sigma\bd x=(x_1,\dots,x_{n-2},x_n,x_{n-1}).
\]
Therefore
\begin{gather*}
\kappa_n\sigma\bd x=\gamma(m;\kappa_{n-1},1)\sigma\bd x
=\kappa_{n-1}(x_1,\dots,x_{n-2},x_n)\oplus x_{n-1}
\\=(\kappa_{n-2}(x_1,\dots,x_{n-2})\oplus x_n)\oplus x_{n-1}.
\end{gather*}
We've already introduced the notation $\xhat$ for $\bd x$ with its last entry removed; now we also need $\xtilde$ for
$\bd x$ with its last two entries removed.  Then what the last calculation shows is that
\[
\kappa_n\sigma\bd x=(\kappa_{n-2}\xtilde\oplus x_n)\oplus x_{n-1},
\]
and we can also expand
\[
\kappa_n\bd x=(\kappa_{n-2}\xtilde\oplus x_{n-1})\oplus x_n.
\]
We can now construct the map $\theta:\kappa_n\bd x\to\kappa_n\sigma\bd x$ as the composite
\[
\xymatrix{
(\kappa_{n-2}\xtilde\oplus x_{n-1})\oplus x_n
\ar[d]^-{\alpha}
\\\kappa_{n-2}\xtilde\oplus(x_{n-1}\oplus x_n)
\ar[d]^-{1\oplus\tau}
\\\kappa_{n-2}\xtilde\oplus(x_n\oplus x_{n-1})
\ar[d]^-{\alpha^{-1}}
\\ (\kappa_{n-2}\xtilde\oplus x_n)\oplus x_{n-1}.
}
\]
Since both $\theta$ and this composite are induced from maps in $Y(n)$, where all maps are unique, this does coincide with $\theta$.

Next, we observe that $\xi_n:\kappa_nF\bd x\to F\kappa_n\bd x$, which expands by definition to
\[
\xymatrix{
\kappa_nF\bd x=(\kappa_{n-1}F\xhat)\oplus Fx_n
\ar[r]^-{\xi_{n-1}\oplus1}
&(F\kappa_{n-1}\xhat)\oplus Fx_n
\ar[r]^-{\xi}
&F(\kappa_{n-1}\xhat\oplus x_n)=F\kappa_n\bd x,
}
\]
expands even further as follows, since $\xi_{n-1}=\xi\circ(\xi_{n-2}\oplus1)$:
\[
\xymatrix{
\kappa_nF\bd x
\ar[d]^-{=}
\\ (\kappa_{n-2}F\xtilde\oplus Fx_{n-1})\oplus Fx_n
\ar[d]^-{(\xi_{n-2}\oplus1)\oplus1}
\\ (F\kappa_{n-2}\xtilde\oplus Fx_{n-1})\oplus Fx_n
\ar[d]^-{\xi\oplus1}
\\F(\kappa_{n-2}\xtilde\oplus x_{n-1})\oplus Fx_n
\ar[d]^-{\xi}
\\F((\kappa_{n-2}\xtilde\oplus x_{n-1})\oplus x_n)
\ar[d]^-{=}
\\F\kappa_n\bd x.
}
\]

We can now fill in the diagram by pasting together two pieces horizontally; the first 
piece is as follows, using the $(\xi_{n-2}\oplus1)\oplus1$ part of the above composition:
\[
\xymatrix@C+30pt{
(\kappa_{n-2}F\xtilde\oplus Fx_n)\oplus Fx_{n-1}
\ar[r]^-{(\xi_{n-2}\oplus1)\oplus1}
\ar[d]_-{\alpha}
&(F\kappa_{n-2}\xtilde\oplus Fx_n)\oplus Fx_{n-1}
\ar[d]^-{\alpha}
\\\kappa_{n-2}F\xtilde\oplus(Fx_n\oplus Fx_{n-1})
\ar[r]^-{\xi_{n-2}\oplus1}
\ar[d]_-{1\oplus\tau}
&F\kappa_{n-2}\xtilde\oplus(Fx_n\oplus Fx_{n-1})
\ar[d]^-{1\oplus\tau}
\\\kappa_{n-2}F\xtilde\oplus(Fx_{n-1}\oplus Fx_n)
\ar[r]^-{\xi_{n-2}\oplus1}
\ar[d]_-{\alpha^{-1}}
&F\kappa_{n-2}\xtilde\oplus(Fx_{n-1}\oplus Fx_n)
\ar[d]^-{\alpha^{-1}}
\\ (\kappa_{n-2}F\xtilde\oplus Fx_{n-1})\oplus Fx_n
\ar[r]_-{(\xi_{n-2}\oplus1)\oplus1}
&(F\kappa_{n-2}\xtilde\oplus Fx_{n-1})\oplus Fx_n.
}
\]
The top and bottom rectangles commute by naturality of $\alpha$, and the middle one by inspection.

The second piece, to be pasted horizontally to the first one, is
\[
\xymatrix@C+10pt{
(F\kappa_{n-2}\xtilde\oplus Fx_{n-1})\oplus Fx_n
\ar[r]^-{\xi\circ(\xi\oplus1)}
\ar[d]_-{\alpha}
&F((\kappa_{n-2}\xtilde\oplus x_{n-1})\oplus x_n)
\ar[d]^-{F\alpha}
\\F\kappa_{n-2}\xtilde\oplus(Fx_{n-1}\oplus Fx_n)
\ar[d]_-{1\oplus\tau}
\ar[r]^-{\xi\circ(1\oplus\xi)}
&F(\kappa_{n-2}\xtilde\oplus(x_{n-1}\oplus x_n))
\ar[d]^-{F(1\oplus\tau)}
\\F\kappa_{n-2}\xtilde\oplus(Fx_n\oplus Fx_{n-1})
\ar[d]_-{\alpha^{-1}}
\ar[r]^-{\xi\circ(1\oplus\xi)}
&F(\kappa_{n-2}\xtilde\oplus(x_n\oplus x_{n-1}))
\ar[d]^-{F\alpha^{-1}}
\\ (F\kappa_{n-2}\xtilde\oplus Fx_n)\oplus Fx_{n-1}
\ar[r]_-{\xi\circ(\xi\oplus1)}
&F((\kappa_{n-2}\xtilde\oplus x_n)\oplus x_{n-1}).
}
\]
The perimeter of the total pasted diagram gives our desired diagram, so it just remains to verify commutativity
of this second piece of the pasting.  The top and bottom rectangles both commute by the coherence diagram
for $\alpha$ and $\xi$, so this just leaves the middle rectangle.  But using generic variables to save space, this
can be expanded to
\[
\xymatrix{
Fa\oplus(Fb\oplus Fc)
\ar[r]^-{1\oplus\xi}
\ar[d]_-{1\oplus \tau}
&Fa\oplus F(b\oplus c)
\ar[r]^-{\xi}
\ar[d]^-{1\oplus F\tau}
&F(a\oplus(b\oplus c))
\ar[d]^-{F(1\oplus\tau)}
\\Fa\oplus(Fc\oplus Fb)
\ar[r]_-{1\oplus\xi}
&Fa\oplus F(c\oplus b)
\ar[r]_-{\xi}
&F(a\oplus(c\oplus b)).
}
\]
However, the left square commutes by the transposition coherence diagram for $\xi$,
and the right square by naturality of $\xi$.
This completes the verification, and so $F$ does
define a multifunctor.
\end{proof}

\section{The Weak Left Adjoint}\label{ladj}

This section is devoted to the proof of Theorem \ref{ladjtheorem}.
As a result of Theorem \ref{laxtheorem}, we have a well-defined underlying multicategory functor 
\[
U:\Sym\to\Mult
\]
with source symmetric monoidal categories and lax monoidal functors, and target multicategories and multifunctors.
It is natural to ask for a left adjoint to this construction, and we almost have one: there is a functor
\[
L:\Mult\to\Sym
\]
with maps $\eta:\id\to UL$ and $\varepsilon:LU\to\id$ satisfying the triangle identities for the unit and counit
of an adjunction.  The only problem is that while $\eta$ is natural, $\varepsilon$ is only lax natural: its naturality squares
only commute up to a natural map of their own that satisfies a coherence condition.  The constructions, which are identical to
those in \cite{EM2}, Theorem 4.2, are as follows.

\begin{construction}\label{constructleftadjoint}
Let $\bd M$ be a multicategory.  We construct a symmetric monoidal category $L\bd M$, which is actually permutative, as follows.
The objects of $L\bd M$ are the free monoid on the objects of $\bd M$, namely
\[
\coprod_{n\ge0}\Ob\bd M^n.
\]
Given objects $\bd x=(x_1,\dots,x_j)$ and $\bd y=(y_1,\dots,y_n)$ of $L\bd M$, we define the elements of the morphism set 
$L\bd M(\bd x,\bd y)$ to consist of ordered pairs $(f,\br{\phi_s})$, where $f:\{1,\dots,j\}\to\{1,\dots,n\}$ is a function with no further
structure, and for $1\le s\le n$, $\phi_s$ is a morphism in $\bd M(\br{x_r}_{f(r)=s};y_s)$.

Given a third object $\bd z=(z_1,\dots,z_p)$ of $L\bd M$, and a morphism $(g,\br{\psi_t})\in L\bd M(\bd y,\bd z)$, we define
\[
(g,\br{\psi_t})\circ(f,\br{\phi_s})=(g\circ f,\br{\chi_t}),
\]
where the morphisms $\chi_t:\br{x_r}_{gf(r)=t}\to z_t$ are given by the composite
\[
\xymatrix@C+20pt{
\br{x_r}_{gf(r)=t}\cong\bigodot_{g(s)=t}\br{x_r}_{f(r)=s}
\ar[r]^-{\odot_{g(s)=t}\phi_s}
&\br{y_s}_{g(s)=t}
\ar[r]^-{\psi_t}
&z_t.
}
\]
Here the first isomorphism simply rearranges the tuple $\br{x_r}_{gf(r)=t}$ into chunks corresponding to each $s$ for which $g(s)=t$.
The symmetric monoidal structure, which is actually permutative, is given by concatenation for the product, and the
empty list as the unit.  

To define the functoriality of this constriction, suppose given a multifunctor $F:\bd M\to\bd N$.  We define $LF:L\bd M\to L\bd N$, which will actually
be a strict monoidal functor, as follows.  On objects, this is just the free monoid functor, so
\[
LF(x_1,\dots,x_j)=(Fx_1,\dots,Fx_j).
\]
On morphisms, suppose given $(f,\br{\phi_s}):\bd x\to\bd y$, where $\bd x$ and $\bd y$ are as above, so we have
\[
f:\{1,\dots,j\}\to\{1,\dots,n\}\text{ and }\phi_s:\br{x_r}_{f(r)=s}\to y_s\text{ for }1\le s\le n.
\]
Then we just define
\[
LF(f,\br{\phi_s})=(f,\br{F\phi_s}).
\]
Functoriality is now straightforward to verify, as is the fact that $LF$ is strict monoidal, and in fact strict symmetric monoidal.
\end{construction}

To show that this construction gives a weak left adjoint to the underlying multicategory construction, we provide a unit and weak counit, and
show that the adjunction triangles commute.  Since $L\bd M$ is permutative, we use the usual definition for the underlying multicategory
of a permutative category, which in this case becomes
\[
UL\bd M(\bd x_1,\dots,\bd x_n;\bd y)=L\bd M(\odot_s\bd x_s,\bd y).
\]

\begin{definition}
Let $\bd M$ be a multicategory.  We define the unit map $\eta:\bd M\to UL\bd M$ as follows.  On objects, this is just
the unit map of the free-forgetful adjunction between sets and monoids, so sends an object $x$ to the list of length 1 with entry $x$,
i.e, we include $\Ob\bd M$ as the objects in level 1 in $\Ob L\bd M$.  On morphisms, suppose $\phi\in\bd M(\bd x;y)$, where
$\bd x=(x_1,\dots,x_j)$.  Then there is exactly one function $p:\{1,\dots,j\}\to\{1\}$, so we send $\phi$ to the $j$-morphism $(p,\{\phi\})$
of $UL\bd M$.  We can now check that this defines a map of multicategories, and gives a natural map $\id\to UL$.
\end{definition}

For the weak counit, we use the following.

\begin{definition}
Let $\C$ be a symmetric monoidal category.  Suppose again that we are using
our constant sequence $\kappa_n:\Ob\C^n\to\Ob Y(n)$ 
with $\kappa_0=0$, $\kappa_1=1$, and $\kappa_n=\gamma(m;\kappa_{n-1},1)$ for $n\ge2$
to define our underlying multicategory.  
We define the counit map $\varepsilon:LU\C\to\C$ as follows.  
On objects our map sends
\[
(y_1,\dots,y_n)=\bd y\mapsto\kappa_n\bd y=\kappabar\bd y.
\]
Now let $\bd x=(x_1,\dots,x_j)$, and suppose given a morphism $(f,\br{\phi_s})\in LU\C(\bd x,\bd y)$, so 
$f:\{1,\dots,j\}\to\{1,\dots,n\}$ and $\phi_s:\br{x_r}_{f(r)=s}\to y_s$ in $U\C$.
Let $j_s=\abs{f^{-1}s}$, so $j=j_1+\cdots+j_n$.  Then by definition, we have
\[
\phi_s\in\C(\kappabar(\br{x_r}_{f(r)=s}),y_s),
\]
and so
\[
\br{\phi_s}\in\C^n(\br{\kappabar(\br{x_r}_{f(r)=s})},\bd y).
\]
Let $\sigma_f$ be the element of $\Sigma_j$ for which
\[
\sigma_f\cdot\bd x=\odot_r\br{x_s}_{f(r)=s}.
\]
Then we have isomorphisms in $Y(j)$
\[
\xymatrix{
\kappa_j
\ar[r]^-{\cong}
&\kappa_j\cdot\sigma_f
\ar[r]^-{\cong}
&\gamma(\kappa_n;\br{\kappa_{j_s}})\cdot\sigma_f,
}
\]
which induce the isomorphisms in the following composition:
\[
\xymatrix{
\C^n(\br{\kappabar(\br{x_r}_{f(r)=s})},\bd y)
\ar[d]^-{\kappa_n}
\\\C(\kappa_n\br{\kappabar(\br{x_r}_{f(r)=s})},\kappabar\bd y)
\ar[d]^-{\cong}
\\\C(\kappabar(\odot_s\br{x_r}_{f(r)=s}),\kappabar\bd y)
\ar[d]^-{\cong}
\\\C(\kappabar\bd x,\kappabar\bd y).
}
\]
We use the image of $\br{\phi_s}$ under this composition as our $\varepsilon(f,\br{\phi_s})$.
It is an exercise to see that  $\varepsilon$ is a strong symmetric monoidal functor.
\end{definition}

Several remarks are in order about this weak counit.  First, it is \emph{not} strictly natural, but rather only natural
up to a natural transformation, in the sense that if we are given a lax symmetric monoidal functor $F:\C\to\D$,
we can form the following diagram:
\[
\xymatrix{
LU\C\ar[r]^-{LUF}\ar[d]_{\varepsilon_\C}
&LU\D\ar[d]^-{\varepsilon_\D}
\\\C\ar[r]_-{F}
&\D.
}
\]
For $\varepsilon$ to be natural, this diagram would have to always commute on the nose, but instead it only
commutes up to a natural map, which is an isomorphism if $F$ is strong monoidal.  If $F$ is strict monoidal, we do
get equality.  We see this since following
an object $(x_1, x_2)$ of $LU\C$ counterclockwise, it ends up first at $x_1\oplus x_2$, and then $F(x_1\oplus x_2)$.
However, clockwise it ends up first at $(Fx_1,Fx_2)$, and then at $Fx_1\oplus Fx_2$.  We then have to use
the lax monoidal structure map
\[
\xi_F:Fx_1\oplus Fx_2\to F(x_1\oplus x_2)
\]
in order to give a map between the two ways of traversing the square.  For a general object $\bd x=(x_1,\dots,x_n)$ of $LU\C$,
we use the map $\xi_n:\kappabar F\bd x\to F\kappabar\bd x$ of Definition 4.5, which is only an equality
when $F$ is a strict monoidal functor.  These then combine to give a
natural map $\xi$ from the clockwise direction to the counterclockwise direction, so we have a 2-cell
\[
\xymatrix{
LU\C\ar[r]^-{LUF}\ar[d]_{\varepsilon_\C}
\drtwocell<\omit>{\,\,\,\xi_F}
&LU\D\ar[d]^-{\varepsilon_\D}
\\\C\ar[r]_-{F}
&\D
}
\]
instead of a commutative diagram.  The natural maps $\xi$ are coherent, in the sense that if we have another lax symmetric
monoidal functor $G:\D\to\E$, then we have the following equality of pasting diagrams:
\[
\xymatrix{
LU\C\ar[r]^-{LUF}\ar[d]_{\varepsilon_\C}
\drtwocell<\omit>{\,\,\,\xi_F}
&LU\D\ar[d]^-{\varepsilon_\D}\ar[r]^-{LUG}
\drtwocell<\omit>{\,\,\,\xi_G}
&LU\E\ar[d]^-{\varepsilon_\E}
\\\C\ar[r]_-{F}
&\D\ar[r]_-{G}
&\E
}
\,\,\,
\xymatrix{
\relax\ar@{}[d]_-{\displaystyle=}
\\\relax
}
\xymatrix@C+30pt{
LU\C\ar[r]^-{LU(G\circ F)}\ar[d]_-{\varepsilon_\C}
\drtwocell<\omit>{\,\,\,\,\,\,\xi_{GF}}
&LU\E\ar[d]^-{\varepsilon_\E}
\\\C\ar[r]_-{G\circ F}
&\E.
}
\]
Further, $\xi_\id$ is the identity transformation.

Note that the counit $\varepsilon$ is strict monoidal if $\C$ is permutative, but only strong monoidal in general: we have
\[
\varepsilon\bd x\oplus\varepsilon\bd y
=\kappabar\bd x\oplus\kappabar\bd y
=\gamma(m;\kappa_j,\kappa_n)\cdot(\bd x\odot\bd y),
\]
but 
\[
\varepsilon(\bd x\odot\bd y)
=\kappabar(\bd x\odot\bd y)
=\kappa_{j+n}\cdot(\bd x\odot\bd y),
\]
and $\kappa_{j+n}\ne\gamma(m;\kappa_j,\kappa_n)$ in $Y({j+n})$.  
In all cases, the structure map $\varepsilon\bd x\oplus\varepsilon\bd y\to\varepsilon(\bd x\odot\bd y)$
is induced by the unique isomorphism
\[
\gamma(m;\kappa_j,\kappa_n)\to\kappa_{j+n}
\]
in $Y(j+n)$,
which induces an identity if $\C$ is permutative, but not in general.

Since we have chosen $\kappa_0=0$, it follows that $\varepsilon$ is strictly unital: $\varepsilon()=\kappa_0*=e_\C$.  

We leave to the reader the verification of the coherence diagrams showing that $\varepsilon$ is a symmetric monoidal functor.

The proof of Theorem \ref{ladjtheorem} concludes with the following lemma.

\begin{lemma}
The adjunction triangles
\[
\xymatrix{
U\C
\ar[r]^-{\eta_U}
\ar[dr]_-{=}
&UL U\C
\ar[d]^-{U\varepsilon}
\\&U\C
}
\text{ and }
\xymatrix{
L\bd M
\ar[r]^-{L\eta}
\ar[dr]_-{=}
&LUL\bd M
\ar[d]^-{\varepsilon_{L\bd M}}
\\&L\bd M
}
\]
both commute.
\end{lemma}

\begin{proof}
For the left triangle, an object in $U\C$ is just an object $x\in\C$, which gets sent to a list with one entry in $ULU\C$, and
assuming $\kappa_1=1$ allows us to conclude that this gets sent in turn back to $\kappa_1x=x$.  The triangle therefore
commutes on objects, and the verification on morphisms is the same.

For the right triangle, an object $\bd x\in\Ob\bd M^j\subset\Ob L\bd M$ gets sent by $L\eta$ to a single list $(\bd x)$ in $LUL\bd M$.
This is then concatenated as a list of lists by $\varepsilon$, but since there's only one list in the list of lists, it goes back to
the original list $\bd x$.  The triangle therefore commutes on objects, and a simple check shows that it also commutes on morphisms.
We therefore do have a weak adjunction between the underlying multicategory functor and the free symmetric monoidal functor,
and the left adjoint actually lands in permutative categories and strict monoidal maps.
\end{proof}

Because both adjunction triangles commute strictly, and $\eta$ is strictly natural, we do get a comonad $LU$ on the category
of symmetric monoidal categories and lax symmetric monoidal functors, although the composite $UL$ is only a lax monad
on the category of multicategories.  

For the proof of Theorem \ref{strict}, we have already observed that the weak left adjoint $L$ gives us permutative categories and 
strict monoidal maps as its output, so the comonad $LU$ on symmetric monoidal categories does convert lax symmetric monoidal maps of 
symmetric monoidal categories into strict monoidal maps of permutative categories.  The remaining claim of the Theorem  is the following
proposition:

\begin{proposition}
The weak counit $\varepsilon:LU\C\to\C$ of the weak adjunction is a homotopy equivalence of categories.
\end{proposition}

\begin{proof}
The first adjunction triangle shows that the unit $\eta$ of the adjunction gives a right inverse for $\varepsilon$, since a map on underlying
multicategories is the same as a lax symmetric monoidal functor on symmetric monoidal categories.  We produce a natural map $\nu:\id\to\eta\circ\varepsilon$ of functors $LU\C\to LU\C$,
which therefore shows that $\varepsilon$ and $\eta$ are inverse homotopy equivalences of categories.  

An object $\bd x\in LU\C$ is sent by $\eta\circ\varepsilon$ to the list of length 1 consisting of $\kappabar\bd x$.  
Let's say that $\bd x=(x_1,\dots, x_j)$.
We define $\nu_{\bd x}:\bd x\to(\kappabar\bd x)$ 
as the map in $LU\C$ given 
by the unique function $f:\{1,\dots,j\}\to\{1\}$, together with the identity $\id_{\kappabar\bd x}$.  It is now easy to check
that $\nu$ is a natural map, which concludes the proof. 
\end{proof}

We end this section by noting that $\nu$ does \emph{not} have an inverse, since $f$ is not a bijection.  However, there is a comparison
map to the strictification construction of Isbell \cite{Isbell}, as described explicitly by May in \cite{Einfinity}, Proposition 4.2.  
May's construction
has as its objects the free monoid on the objects of $\C$, but subject to the relation that $e_\C=()$:
the implication is that $e_\C$ is a strict unit to begin with.  
However, the construction can be easily modified to use the free monoid with no relations, which are precisely the objects
of $LU\C$. The morphisms are then created by the counit
$\varepsilon$, rather than being those in Definition \ref{constructleftadjoint}. 
This still gives us a permutative category and a categorical equivalence with the original symmetric monoidal category.
Further, the counit $\varepsilon$ factors through this construction, and therefore $LU\C$ is also homotopy equivalent to it.

\section{Underlying Multicategories: Properties}\label{properties}
We return finally to verifying that the structure given in Section \ref{structure} does satisfy 
the properties necessary
to form a multicategory. 
We return to our original assumption of an arbitrary sequence of functions
\[
\kappa_n:\Ob\C^n\to\Ob Y(n),
\]
which determine the underlying multicategory $U_\kappa\C$.

To give an example of how involved such a structure can look like, we can consider the free symmetric monoidal
category on one object which we describe as follows, using
some notation from \cite{SMC1}. (This example is unnecessary for the subsequent arguments.)  Let's call this 
category $H$.  

The category $H$ is a disjoint union of categories $H(n)$ for $n\ge0$,
where the index $n$ indicates the number of times the freely chosen object is combined with itself. Each category
$H(n)$ has as its objects a set $Z(n)$.
The sets $Z(n)$ themselves consist of ordered pairs
$(a,R)$, where $a$ is a complete parenthesization of at least $n$ letters, 
and $R$ indicates the slots into which $n$ copies of the 
free object are to be inserted, with identities in the rest.  To be precise, we define the parenthesization sets
$V(k)$ for $k\ge0$ by $V(0)=\emptyset$, $V(1)=\{1\}$, and for $k\ge2$,
\[
V(k)=\coprod_{i+j=k}V(i)\times V(j).
\]
The idea is that a complete parenthesization of $k$ letters has a last product, with the left factor and right factors
being parenthesizations of smaller numbers of letters.  The component $a$ of $(a,R)$ is to be an element of
$V(k)$ for $k\ge n$.  

In order to specify the $n$ slots into which our free object is to be inserted, we define 
\[
\mathcal{P}_n(k)=\{R\subset\{1,\dots,k\}:\abs{R}=n\}.
\]
Then we define our object set $Z(n)$ of the component category $H(n)$ by
\[
Z(n)=\coprod_{k\ge n}V(k)\times\mathcal{P}_n(k).
\]

For morphisms, if $(a,R)\in Z(n)$ and $(b,S)\in Z(m)$, then there are no morphisms unless $n=m$,
and if $n=m$, the morphism set is a copy of the symmetric group $\Sigma_n$.  Composition
is given by group multiplication.

The symmetric monoidal structure is given as follows.  First, the unit object is $(1,\emptyset)\in Z(0)$.
Next, for the monoidal product, suppose given $(a,R)\in V(k)\times\mathcal{P}_n(k)\subset Z(n)$ 
and $(b,S)\in V(q)\times\mathcal{P}_m(q)\subset Z(m)$.  Then we define
\[
(a,R)\oplus(b,S)=((a,b),R\amalg k+S),
\]
Here $(a,b)\in V(k)\times V(q)\subset V(k+	q)$, and $R\amalg k+S$ is the subset of $\{1,\dots,k+q\}$
consisting of $R$ together with $k$ added to each element of $S$, so this is just the concatenation
of $R$ and $S$ in $\{1,\dots,k\}\amalg\{1,\dots,q\}=\{1,\dots,k+q\}$.  We therefore obtain an object of $Z(n+m)$.
The monoidal product of morphisms is given by the block sum
$\Sigma_n\times\Sigma_m\to\Sigma_{n+m}$.

The unit isomorphisms and the associator are all given by identity elements of the appropriate $\Sigma_n$.
In detail, if given $(a,R)\in V(k)\times\mathcal{P}_n(k)\subset Z(n)$ then we have
\[
(a,R)\oplus(1,\emptyset)=((a,1),R)\in V(k+1)\times\mathcal{P}_n(k+1)\subset Z(n)
\]
and
\[
(1,\emptyset)\oplus(a,R)=((1,a),1+R)\in V(k+1)\times\mathcal{P}_n(k+1)\subset Z(n).
\]
In both cases, we use the identity element $1_n\in\Sigma_n$ as the unit isomorphism to $(a,R)$.

For the associator, suppose given 
\begin{align*}
&(a,R)\in V(k)\times\mathcal{P}_n(k)\subset Z(n), \\
&(b,S)\in V(q)\times\mathcal{P}_m(q)\subset Z(m), \text{ and }\\
&(c,T)\in V(w)\times\mathcal{P}_t(w)\subset Z(t).
\end{align*}
Then we have
\begin{align*}
&((a,R)\oplus(b,S))\oplus(c,T)=(((a,b),c),R\amalg k+S\amalg k+q+T),\text{ and}\\
&(a,R)\oplus((b,S)\oplus(c,T))=((a,(b,c)),R\amalg k+S\amalg k+q+T).
\end{align*}
Note that
the first of these has first component an element of $V(k+q)\times V(w)$,
while the second has first component an element of $V(k)\times V(q+w)$.  
However, both objects are elements of $Z(n+m+t)$, so we
can use the identity element
$1_{n+m+t}\in\Sigma_{n+m+t}$ as the associator giving an isomorphism between them.

For the transposition, again suppose given $(a,R)\in V(k)\times\mathcal{P}_n(k)\subset Z(n)$ and 
$(b,S)\in V(q)\times\mathcal{P}_m(q)\subset Z(m)$.  Then 
\begin{align*}
&(a,R)\oplus(b,S)=((a,b),R\amalg k+S),\text{ and}\\
&(b,S)\oplus(a,R)=((b,a),S\amalg q+R).
\end{align*}
Both are elements of $Z(n+m)$, but we do not use the identity of $\Sigma_{n+m}$ as the transposition isomorphism, but rather the
element $\tau\br{n,m}$ that transposes a block of length $n$ and a block of length $m$. (We don't want the identity element, since
that would actually give the identity as the transposition when $(b,S)=(a,R)$.)

All the coherence diagrams that do not involve the transposition commute because all the maps are given by identity elements
of a symmetric group.  For the two that do involve the transposition, we find that $\tau^2=\id$ because $\tau\br{n,m}\circ\tau\br{m,n}=\id$.
And the hexagon 
\[
\xymatrix{
((a,R)\oplus(b,S))\oplus(c,T)
\ar[r]^-{\alpha}
\ar[d]_-{\tau\oplus1}
&(a,R)\oplus((b,S)\oplus(c,T))
\ar[r]^-{\tau}
&((b,S)\oplus(c,T))\oplus(a,R)
\ar[d]^-{\alpha}
\\((b,S)\oplus(a,R))\oplus(c,T)
\ar[r]_-{\alpha}
&(b,S)\oplus((a,R)\oplus(c,T))
\ar[r]_-{1\oplus\tau}
&(b,S)\oplus((c,T)\oplus(a,R))
}
\]
commutes because $(1_m\oplus\tau\br{n,t})\circ(\tau\br{n,m}\oplus1_t)=\tau\br{n,m+t}$.

We can now give an example of a sequence of functions $\{\kappa_n: \Ob H^n\to\Ob Y(n)\}$ as follows.
This example is deliberately complicated in order to show how arbitrary such a choice can be.
Suppose given an object $((a_1,R_1),\dots,(a_n,R_n))$ of $H^n$, where $a_i\in Z(j_i)$ and $R_i\in\mathcal{P}_{k_i}(j_i)$,
so $(a_i,R_i)\in\Ob H(k_i)$, and $j_i\ge k_i$ for $1\le i\le n$.  Pick an arbitrary $\beta\in Z(n)$.  Then using the (non-symmetric)
operad structure on $\{Z(n)\}$ given in \cite{SMC1}, we can form
\[
\delta=\gamma(\beta;\alpha_1,\dots,\alpha_n)\in Z(j_1+\cdots j_n).
\]
This replaces each parenthesized slot in $\beta$ with the parenthesization given by the corresponding $\alpha_i$.
Since $j_i>0$ for all $i$, we must have $j_1+\cdots j_n\ge n$, so we can form an object of $Y(n)$ with $\delta$ as its first
component.  Next, let 
\[
S=\{1,j_1+1,j_1+j_2+1,\dots,j_1+\cdots j_{n-1}+1\}\subset\{1,\dots,j_1+\cdots j_n\},
\]
so $S\in\mathcal{P}_n(j_1+\cdots j_n)$.  Finally, to specify an object of $Y(n)$, we need an element of $\Sigma_n$,
so assuming $n\ge2$, we pick the permutation $\sigma$ that transposes the first two elements 
of $\{1,\dots,n\}$
if $n$ is odd, and the last
two elements if $n$ is even.  
Then we assign the object $(\delta,S,\sigma)$ to the object $((\alpha_1,R_1),\dots,(\alpha_n,R_n))$ of $H^n$ in order
to define our function $\kappa_n$.

We now proceed with the proof that $U_\kappa\C$ does have the structure of a multicategory in all cases,
including the example of $U_\kappa H$ for the $H$ and $\kappa$ given above.
We first show that we really do have a right action of $\Sigma_n$ on
the collection of $n$-morphisms in $U_\kappa\C$, and then proceed to verify the 
diagrams given in Definition 2.1 of \cite{EM1}.

\begin{proposition}
Let $\bd x=(x_1,\dots,x_n)\in\Ob\C^n$ and $y\in\Ob\C$.  Then
the maps
\[
\sigma^*:U_\kappa\C(\bd x;y)\to U_\kappa\C(\sigma^{-1}\bd x;y)
\]
produce a right action of $\Sigma_n$ on the collection of $n$-morphisms of $U_\kappa\C$.
\end{proposition}

\begin{proof}
We must show that $1\in\Sigma_n$ produces the identity map, and that given $\sigma,\tau\in\Sigma_n$,
we have
\[
\tau^*\circ\sigma^*=(\sigma\circ\tau)^*.
\]
For both of these, recall that $\sigma^*$ is induced by the unique map in $Y(n)$
\[
\theta(\bd x,\sigma):\kappa_n(\sigma^{-1}\bd x)\to\kappa_n(\bd x)\cdot\sigma.
\]
But if $\sigma=1$, then source and target are both just $\kappa_n\bd x\in Y(n)$, and objects of $Y$
have only the identity as automorphisms, so $\theta(\bd x,1)=\id$, and therefore $1\in\Sigma_n$ induces the identity
on the collection of $n$-morphisms of $U_\kappa\C$.

To see that $\tau^*\circ\sigma^*=(\sigma\circ\tau)^*$, we examine the following diagram between representing
objects:
\[
\xymatrix@C-15pt
{
\kappa_n(\bd x)(\bd x)\ar[r]^-{=}
\ar[d]_-{=}
&[\kappa_n(\bd x)\cdot\sigma](\sigma^{-1}\bd x)
\ar[dd]^-{=}
\ar[r]^-{\theta(\bd x,\sigma)}
&[\kappa_n(\sigma^{-1}\bd x)](\sigma^{-1}\bd x)
\ar[d]^-{=}
\\[\kappa_n(\bd x)\cdot(\sigma\circ\tau)]((\sigma\circ\tau)^{-1}\bd x)
\ar[dr]^-{=}
\ar[dd]_-{\theta(\bd x,\sigma\circ\tau)}
&&[\kappa_n(\sigma^{-1}\bd x)\cdot\tau](\tau^{-1}\sigma^{-1}\bd x)
\ar[dd]^-{\theta(\sigma^{-1}\bd x,\tau)}
\\&[(\kappa_n(\bd x)\cdot\sigma)\cdot\tau](\tau^{-1}\sigma^{-1}\bd x)
\ar[ur]^-{\theta(\bd x,\sigma)\cdot\tau}
\\\kappa_n((\sigma\circ\tau)^{-1}\bd x)((\sigma\circ\tau)^{-1}\bd x)
\ar[rr]_-{=}
&&\kappa_n(\tau^{-1}\sigma^{-1}\bd x)(\tau^{-1}\sigma^{-1}\bd x).
}
\]
The perimeter of the diagram gives the desired identity.  The top left square commutes, being a diagram of identities
on the same object.  The lower pentagon commutes since it is induced from a diagram in $Y(n)$, where all diagrams commute.
This leaves the top right square, which commutes due to the equivariance of the action map of $Y$ on $\C$.  In particular, if we
explicitly say that $\xi:Y(n)\to\Cat(\C^n,\C)$ is the action map, then we have
\[
\xi(\theta(\bd x,\sigma)\cdot\tau)=\xi(\theta(\bd x,\sigma))\cdot\tau.
\]
This is an identity of natural transformations of functors $\C^n\to\C$, which we evaluate at the object
$
\tau^{-1}\sigma^{-1}\bd x.
$
Then the left side becomes
\[
(\theta(\bd x,\sigma)\cdot\tau)\cdot(\tau^{-1}\sigma^{-1}\bd x)
\]
which is the bottom arrow in the top right square, while the right side becomes
\[
\theta(\bd x,\sigma)\cdot\tau\cdot\tau^{-1}\sigma^{-1}\bd x=\theta(\bd x,\sigma)\cdot\sigma^{-1}\bd x,
\]
which is the top arrow in the top right square, so that square does commute.  The entire diagram therefore commutes.
\end{proof}

We must verify the associativity diagram, given as (1) on p.\ 168 of \cite{EM1}, but we give some notation that will
allow us to display it in slightly more compressed form.  So suppose given a final target $d$.  We suppose given a tuple
$\bd c=(c_1,\dots,c_n)$ that will map to $d$, and for each $s$ with $1\le s\le n$, a tuple $\bd b_s$ that will map to $c_s$.
Further,
we write the entries in $\bd b_s$ as $b_{st}$ for a second index $t$,
where we say $1\le t\le j_s$, with $j=j_1+\cdots+j_n$.  
We write the concatenation of all the $\bd b_s$'s as $\odot_s\bd b_s$, which is a $j$-tuple. 
For each index pair $st$, 
we assume given
a tuple $\bd a_{st}$ that will map to $b_{st}$.  
For any fixed $s$ we write the concatenation of the tuples $\bd a_{st}$
as $\odot_t\bd a_{st}$, and the concatenation of all these (so a concatenation of concatenations) as $\odot_s\odot_t\bd a_{st}$.
We now claim the following, where $\Gamma$ is the multiproduct (or composition) in $U_\kappa\C$:

\begin{proposition}
The following associativity diagram in $U_\kappa\C$ commutes:
\[
\xymatrix@C=-72pt @R-5pt
{
&U_\kappa\C(\bd c;d)\times\Prod_s U_\kappa\C(\odot_t\bd a_{st};c_s)
\ar[ddr]^-{\multprod}
\\ U_\kappa\C(\bd c;d)\times\Prod_s\left(U_\kappa\C(\bd b_s;c_s)\times\Prod_t U_\kappa\C(\bd a_{st};b_{st})\right) 
\ar[ur]^-{\id\times\prod_s\multprod}\ar[dd]_-{\cong}
\\&& U_\kappa\C(\odot_s\odot_t\bd a_{st};d). 
\\ U_\kappa\C(\bd c;d)\times\Prod_sU_\kappa\C(\bd b_s;c_s)\times\Prod_s\Prod_tU_\kappa\C(\bd a_{st};b_{st})
\ar[dr]^-{\multprod\times1}
\\& U_\kappa\C(\odot_s\bd b_s;d)\times\Prod_s\Prod_tU_\kappa\C(\bd a_{st};b_{st})
\ar[uur]_-{\multprod}
}
\]
\end{proposition}

\begin{proof}
The basic idea for showing that this diagram commutes is to connect it to the associativity square in $\C$
displayed below:
\[
\xymatrix@C=-172pt{
&\C(\kappabar(\bd c),d)\times\C(\kappabar(\odot_s\bd b_s),\kappabar(\bd c))\times\C(\kappabar(\odot_s\odot_t\bd a_{st}),\kappabar(\odot_s\bd b_s))
\ar[dl]_-{\circ\times1}\ar[ddr]^-{1\times\circ}
\\\C(\kappabar(\odot_s\bd b_s),d)\times\C(\kappabar(\odot_s\odot_t\bd a_{st}),\kappabar(\odot_s\bd b_s))
\ar[ddr]_-{\circ}
\\&&\C(\kappabar(\bd c),d)\times\C(\kappabar(\odot_s\odot_t\bd a_{st}),\kappabar(\bd c))
\ar[dl]^-{\circ}
\\&\C(\kappabar(\odot_s\odot_t\bd a_{st}),d).
}
\]

Unpacking the counterclockwise direction of the desired associativity diagram using the definition of the multiproduct composition $\Gamma$,
we get the following:
\[
\xymatrix{
\C(\kappabar\bd c,d)\times\C^n(\br{\kappabar\bd b_s},\bd c)\times\C^{j}(\odot_s\br{\kappabar\bd a_{st}},\odot_s\bd b_s)
\ar[d]^-{1\times\kappa_n(\bd c)\times1}
\\\C(\kappabar\bd c,d)\times\C(\kappa_n(\bd c)\br{\kappabar\bd b_s},\kappabar\bd c)\times\C^{j}(\odot_s\br{\kappabar\bd a_{st}},\odot_s\bd b_s)
\ar[d]^-{\circ\times1}
\\\C(\kappa_n(\bd c)\br{\kappabar\bd b_s},d)\times\C^{j}(\odot_s\br{\kappabar\bd a_{st}},\odot_s\bd b_s)
\ar[d]^-{\C(\phi(\bd c,\br{\bd b_s}),1)\times1}
\\\C(\kappabar(\odot_s\bd b_s),d)\times\C^{j}(\odot_s\br{\kappabar\bd a_{st}},\odot_s\bd b_s)
\ar[d]^-{1\times\kappa_j(\odot_s\bd b_s)}
\\\C(\kappabar(\odot_s\bd b_s),d)\times\C(\kappa_j(\odot_s\bd b_s)(\odot_s\br{\kappabar\bd a_{st}}),\kappabar(\odot_s\bd b_s))
\ar[d]^-{\circ}
\\\C(\kappa_j(\odot_s\bd b_s)(\odot_s\br{\kappabar\bd a_{st}}),d)
\ar[d]^-{\C(\phi(\odot_s\bd b_s,\br{\odot_s\bd a_{st}}),1)}
\\\C(\kappabar(\odot_s\odot_t\bd a_{st}),d).
}
\]
However, the composites and the actions of the morphisms in $Y$ are essentially independent, so we can compress the
display to the following:
\[
\xymatrix{
\C(\kappabar\bd c,d)\times\C^n(\br{\kappabar\bd b_s},\bd c)\times\C^{j}(\odot_s\br{\kappabar\bd a_{st}},\odot_s\bd b_s)
\ar[d]^-{1\times\kappa_n(\bd c)\times\kappa_j(\odot_s\bd b_s)}
\\\makebox[.7\textwidth][c]{$\C(\kappabar\bd c,d)\times\C(\kappa_n(\bd c)\br{\kappabar\bd b_s},\kappabar\bd c)
\times\C(\kappa_j(\odot_s\bd b_s)(\odot_s\br{\kappabar\bd a_{st}}),\kappabar(\odot_s\bd b_s))$}
\ar[d]^-{1\times\C(\phi(\bd c,\br{\bd b_s}),1)\times\C(\phi(\odot_s\bd b_s,\br{\odot_s\bd a_{st}}),1)}
\\\C(\kappabar\bd c,d)\times\C(\kappabar(\odot_s\bd b_s),\kappabar\bd c)\times\C(\kappabar(\odot_s\odot_t\bd a_{st}),\kappabar(\odot_s\bd b_s))
\ar[d]^-{\circ\times1}
\\\C(\kappabar(\odot_s\bd b_s),d)\times\C(\kappabar(\odot_s\odot_t\bd a_{st}),\kappabar(\odot_s\bd b_s))
\ar[d]^-{\circ}
\\\C(\kappabar(\odot_s\odot_t\bd a_{st}),d).
}
\]
In particular, the last two maps are now the counterclockwise direction in the associativity diagram we know commutes in $\C$.

We turn next to the clockwise direction in the desired associativity diagram.  Using the definition of the multiproduct in $U_\kappa\C$,
it unpacks as follows:
\[
\xymatrix{
\C(\kappabar\bd c,d)\times\C^n(\br{\kappabar\bd b_s},\bd c)\times\C^{j}(\odot_s\br{\kappabar\bd a_{st}},\odot_s\bd b_s)
\ar[d]^-{1\times1\times\odot_s\kappa_{j_s}\bd b_s}
\\\C(\kappabar\bd c,d)\times\C^n(\br{\kappabar\bd b_s},\bd c)\times\C^n(\br{\kappa_{j_s}\bd b_s\br{\kappabar\bd a_{st}}},\br{\kappabar\bd b_s})
\ar[d]^-{1\times\circ}
\\\C(\kappabar\bd c,d)\times\C^n(\br{\kappa_{j_s}\bd b_s\br{\kappabar\bd a_{st}}},\bd c)
\ar[d]^-{1\times\C^n(\br{\phi(\bd b_s,\br{\bd a_{st}})},1)}
\\\C(\kappabar\bd c,d)\times\C^n(\br{\kappabar(\odot_t\bd a_{st})},\bd c)
\ar[d]^-{1\times\kappa_n\bd c}
\\\C(\kappabar\bd c,d)\times\C(\kappa_n\bd c\br{\kappabar(\odot_t\bd a_{st})},\kappabar\bd c)
\ar[d]^-{\circ}
\\\C(\kappa_n\bd c\br{\kappabar(\odot_t\bd a_{st})},d)
\ar[d]^-{\C(\phi(\bd c,\br{\odot_t\bd a_{st}}),1)}
\\\C(\kappabar(\odot_s\odot_t\bd a_{st}),d).
}
\]
We can rewrite this direction also to delay the compositions to the end, 
which we omit and will paste on in the next display. We get the following:
\[
\xymatrix{
\C(\kappabar\bd c,d)\times\C^n(\br{\kappabar\bd b_s},\bd c)\times\C^{j}(\odot_s\br{\kappabar\bd a_{st}},\odot_s\bd b_s)
\ar[d]^-{1\times1\times\odot_s\kappa_{j_s}\bd b_s}
\\\C(\kappabar\bd c,d)\times\C^n(\br{\kappabar\bd b_s},\bd c)\times\C^n(\br{\kappa_{j_s}\bd b_s\br{\kappabar\bd a_{st}}},\br{\kappabar\bd b_s})
\ar[d]^-{1\times1\times\C^n(\br{\phi(\bd b_s,\br{\bd a_{st}})},1)}
\\\C(\kappabar\bd c,d)\times\C^n(\br{\kappabar\bd b_s},\bd c)\times\C^n(\br{\kappabar(\odot_t\bd a_{st})},\br{\kappabar\bd b_s})
\ar[d]^-{1\times\kappa_n\bd c\times\kappa_n\bd c}
\\\C(\kappabar\bd c,d)\times\C(\kappa_n\bd c\br{\kappabar\bd b_s},\kappabar\bd c)
\times\C(\kappa_n\bd c\br{\kappabar(\odot_t\bd a_{st})},\kappa_n\bd c\br{\kappabar\bd b_s})
\ar[d]^-{1\times1\times\C(\phi(\bd c,\br{\odot_t\bd a_{st}}),1)}
\\\C(\kappabar\bd c,d)\times\C(\kappa_n\bd c\br{\kappabar\bd b_s},\kappabar\bd c)
\times\C(\kappabar(\odot_s\odot_t\bd a_{st}),\kappa_n\bd c\br{\kappabar\bd b_s}).
}
\]
We now insert an extra arrow that has no effect on the composite, but allows us to connect the display to the associativity
diagram in $\C$: we paste the following to the bottom of the previous display, giving the clockwise direction in the desired diagram:
\[
\xymatrix{
\\\C(\kappabar\bd c,d)\times\C(\kappa_n\bd c\br{\kappabar\bd b_s},\kappabar\bd c)
\times\C(\kappabar(\odot_s\odot_t\bd a_{st}),\kappa_n\bd c\br{\kappabar\bd b_s}).
\ar[d]^-{1\times\C(\phi(\bd c,\br{\bd b_s}),1)\times\C(1,\phi(\bd c,\br{\bd b_s})^{-1})}
\\\C(\kappabar\bd c,d)\times\C(\kappabar(\odot_s\bd b_s),\kappabar\bd c)
\times\C(\kappabar(\odot_s\odot_t\bd a_{st}),\kappabar(\odot_s\bd b_s))
\ar[d]^-{1\times\circ}
\\\C(\kappabar\bd c,d)\times\C(\kappabar(\odot_s\odot_t\bd a_{st}),\kappabar\bd c)
\ar[d]^-{\circ}
\\\C(\kappabar(\odot_s\odot_t\bd a_{st}),d).
}
\]
Note that the first arrow in the above display composes with $\phi(\bd c,\br{\bd b_s})$ and its inverse, which then get
composed together at the next step.  The new arrow therefore has no effect on the composite, so we do actually still have the
same clockwise direction in our desired diagram, but the last two arrows now coincide with the clockwise direction
in the associativity diagram for $\C$.
It therefore suffices to show that the two directions coincide at the point that they reach the known associativity diagram.

Another benefit of rewriting the two directions to delay compositions to the end is that we can follow what happens to each
of the three factors in the beginning term
\[
\C(\kappabar\bd c,d)\times\C^n(\br{\kappabar\bd b_s},\bd c)\times\C^{j}(\odot_s\br{\kappabar\bd a_{st}},\odot_s\bd b_s)
\]
independently, since the two maps to the known associativity diagram are products of three maps from each of the three
factors.  
Starting with the first term $\C(\kappabar\bd c,d)$, we see that nothing happens to it in either direction,
so the two directions do coincide on that factor.  

The maps in either direction on the second factor consist of the same composite, namely
\[
\xymatrix{
\C^n(\br{\kappabar\bd b_s},\bd c)
\ar[r]^-{\kappa_n\bd c}
&\C(\kappa_n\bd c\br{\kappabar\bd b_s},\kappabar\bd c)
\ar[rr]^-{\C(\phi(\bd c,\br{\bd b_s}),1)}
&&\C(\kappabar(\odot_s\bd b_s),\kappabar\bd c).
}
\]
The two directions therefore coincide on that factor as well, reducing the issue to the restrictions to the third 
factor in either direction.  

On the third factor, the counterclockwise direction restricts as follows:
\[
\xymatrix{
\C^{j}(\odot_s\br{\kappabar\bd a_{st}},\odot_s\bd b_s)
\ar[d]^-{\kappa_j(\odot_s\bd b_s)}
\\\C(\kappa_j(\odot_s\bd b_s)(\odot_s\br{\kappabar\bd a_{st}}),\kappabar(\odot_s\bd b_s))
\ar[d]^-{\C(\phi(\odot_s\bd b_s,\br{\odot_t\bd a_{st}}),1)}
\\\C(\kappabar(\odot_s\odot_t\bd a_{st}),\kappabar(\odot_s\bd b_s)).
}
\]
However, the clockwise direction restricts as
\[
\xymatrix{
\C^{j}(\odot_s\br{\kappabar\bd a_{st}},\odot_s\bd b_s)
\ar[d]^-{\odot_s\kappa_{j_s}\bd b_s}
\\\C^n(\br{\kappa_{j_s}\bd b_s\br{\kappabar\bd a_{st}}},\br{\kappabar\bd b_s})
\ar[d]^-{\C^n(\br{\phi(\bd b_s,\br{\bd a_{st}})},1)}
\\\C^n(\br{\kappabar(\odot_t\bd a_{st})},\br{\kappabar\bd b_s})
\ar[d]^-{\kappa_n\bd c}
\\\C(\kappa_n\bd c\br{\kappabar(\odot_t\bd a_{st})},\kappa_n\bd c\br{\kappabar\bd b_s})
\ar[d]^-{\C(\phi(\bd c,\br{\bd a_{st}}),\phi(\bd c,\br{\bd b_s})^{-1})}
\\\C(\kappabar(\odot_s\odot_t\bd a_{st}),\kappabar(\odot_s\bd b_s)).
}
\]
Because of the functoriality of $\kappa_n\bd c$, we can rearrange the middle two arrows, so
the array becomes
\[
\xymatrix{
\C^{j}(\odot_s\br{\kappabar\bd a_{st}},\odot_s\bd b_s)
\ar[d]^-{\odot_s\kappa_{j_s}\bd b_s}
\\\C^n(\br{\kappa_{j_s}\bd b_s\br{\kappabar\bd a_{st}}},\br{\kappabar\bd b_s})
\ar[d]^-{\kappa_n\bd c}
\\\C(\kappa_n\bd c\br{\kappa_{j_s}\bd b_s\br{\kappabar\bd a_{st}}},\kappa_n\bd c\br{\kappabar\bd b_s})
\ar[d]^-{\C(\kappa_n\bd c\br{\phi(\bd b_s,\br{\bd a_{st}})},1)}
\\\C(\kappa_n\bd c\br{\kappabar(\odot_t\bd a_{st})},\kappa_n\bd c\br{\kappabar\bd b_s})
\ar[d]^-{\C(\phi(\bd c,\br{\odot_t\bd a_{st}}),\phi(\bd c,\br{\bd b_s})^{-1})}
\\\C(\kappabar(\odot_s\odot_t\bd a_{st}),\kappabar(\odot_s\bd b_s)),
}
\]
or even more compactly as
\[
\xymatrix{
\C^{j}(\odot_s\br{\kappabar\bd a_{st}},\odot_s\bd b_s)
\ar[d]^-{\gamma(\kappa_n\bd c;\br{\kappa_{j_s}\bd b_s})}
\\\C(\kappa_n\bd c\br{\kappa_{j_s}\bd b_s\br{\kappabar\bd a_{st}}},\kappa_n\bd c\br{\kappabar\bd b_s})
\ar[d]^-{\C(\kappa_n\bd c\br{\phi(\bd b_s,\br{\bd a_{st}})},1)}
\\\C(\kappa_n\bd c\br{\kappabar(\odot_t\bd a_{st})},\kappa_n\bd c\br{\kappabar\bd b_s})
\ar[d]^-{\C(\phi(\bd c,\br{\odot_t\bd a_{st}}),\phi(\bd c,\br{\bd b_s})^{-1})}
\\\C(\kappabar(\odot_s\odot_t\bd a_{st}),\kappabar(\odot_s\bd b_s)).
}
\]
The last two arrows are induced by isomorphisms in the operad $Y$, which are entirely determined by
their sources and targets, so lets just indicate them by isomorphism symbols $\cong$.  The same is true
for the second arrow from the counterclockwise direction.  Then asking that the two directions coincide is the
same as asking for the commutativity of the following square:
\[
\xymatrix@C+10pt{
\C^{j}(\odot_s\br{\kappabar\bd a_{st}},\odot_s\bd b_s)
\ar[d]_-{\kappa_j(\odot_s\bd b_s)}
\ar[r]^-{\gamma(\kappa_n\bd c;\br{\kappa_{j_s}\bd b_s})}
&\C(\kappa_n\bd c\br{\kappa_{j_s}\bd b_s\br{\kappabar\bd a_{st}}},\kappa_n\bd c\br{\kappabar\bd b_s})
\ar[d]^-{\cong}
\\\C(\kappa_j(\odot_s\bd b_s)(\odot_s\br{\kappabar\bd a_{st}}),\kappabar(\odot_s\bd b_s))
\ar[r]_-{\cong}
&\C(\kappabar(\odot_s\odot_t\bd a_{st}),\kappabar(\odot_s\bd b_s)).
}
\]

If we trace a typical element $f\in\C^j(\odot_s\br{\kappabar\bd a_{st}},\odot_s\bd b_s)$ around the square, we find that
we are asking for commutativity of the following diagram:
\[
\xymatrix@C+35pt{
&\kappa_n\bd c\br{\kappa_j\bd b_s(\odot_s\br{\kappabar\bd a_{st}})}
\ar[r]^-{\gamma(\kappa_n\bd c\br{\kappa_{j_s}\bd b_s})f}
&\kappa_n\bd c\br{\kappabar\bd b_s}
\\\kappabar(\odot_s\odot_t\bd a_{st})
\ar[ur]^-{\cong}
\ar[dr]_-{\cong}
\\&\kappa_j(\odot_s\bd b_s)(\odot_s\br{\kappabar\bd a_{st}})
\ar[uu]_-{\phi(\bd c,\odot_s\bd b_s)}^-{\cong}
\ar[r]_-{\kappa_j(\odot_s\bd b_s)f}
&\kappabar(\odot_s\bd b_s).
\ar[uu]_-{\phi(\bd c,\odot_s\bd b_s)}^-{\cong}
}
\]
The square commutes by naturality of $\phi(\bd c,\odot_s\bd b_s)$, and the triangle because it is
induced by a diagram in $Y$, however not in $Y(j)$, but rather in the degree of the total concatenation
$\odot_s\odot_t\bd a_{st}$.  The diagram therefore does commute, and the composition multiproduct
in $U_\kappa\C$ is associative.
\end{proof}

Continuing the verification of diagrams from \cite{EM1}, we turn to the unit diagrams, listed as (2) on p.\ 168 in \cite{EM1}.  We
begin with the first of them:

\begin{proposition}
The following unit diagram in $U_\kappa\C$ commutes:
\[
\xymatrix{
U_\kappa\C(\bd c;d)\times\{1\}^n
\ar[r]^-{\cong}
\ar[d]_-{\id\times1^n}
&U_\kappa\C(\bd c;d).
\\U_\kappa\C(\bd c;d)\times\prod_{s=1}^n U_\kappa\C(c_s;c_s)
\ar[ur]_-{\Gamma}
}
\]
\end{proposition}

\begin{proof}
Expanding the counterclockwise direction in the diagram using the definitions of the structure maps, we have
\[
\xymatrix{
\C(\kappabar\bd c,d)\times\{1\}^n
\ar[d]^-{\id\times1^n}
\\\C(\kappabar\bd c,d)\times\prod_{s=1}^n\C(\kappabar c_s,c_s)
\ar[d]^-{=}
\\\C(\kappabar\bd c,d)\times\C^n(\br{\kappabar c_s},\bd c)
\ar[d]^-{\id\times\kappa_n\bd c}
\\\C(\kappabar\bd c,d)\times\C(\kappa_n\bd c\br{\kappabar c_s},\kappabar\bd c)
\ar[d]^-{\id\times\C(\phi(\bd c,\br{c_s}),1)}
\\\C(\kappabar\bd c,d)\times\C(\kappabar\bd c,\kappabar\bd c)
\ar[d]^-{\circ}
\\\C(\kappabar\bd c,d).
}
\]
Since we want this to coincide with the canonical projection, and nothing happens to the first factor until the composition
at the end, it suffices to trace the second factor and see that it lands at $\id_{\kappabar\bd c}\in\C(\kappabar\bd c,\kappabar\bd c)$
at the next to the last step.

To see this, note first that the 1's get sent to the maps induced by the unique maps $\omega(c_s):\kappa_1(c_s)\to c_s$ in $Y(1)$, by definition
of the unit map.  Next, observe that
\[
\kappa_n\bd c\br{\kappabar c_s}=\gamma(\kappa_n\bd c;\br{\kappa_1c_s})\bd c,
\]
and the map from there to $\kappabar\bd c$ is the one induced by the maps $\omega(c_s):\kappa_1(c_s)\to 1$ in $Y(1)$.  But this also gives
the inverse to the unique isomorphism
\[
\phi(\bd c,\br{c_s}):\kappa_n\bd c\to\gamma(\kappa_n\bd c;\br{\kappa_1c_s}),
\]
so we do end up at the identity element at the next to the last step.  The claimed diagram therefore does commute.
\end{proof}

We next verify the second unit diagram.

\begin{proposition}
The unit diagram
\[
\xymatrix{
\{1\}\times U_\kappa\C(\bd c;d)
\ar[r]^-{\cong}
\ar[d]_-{1\times\id}
&U_\kappa\C(\bd c;d)
\\U_\kappa\C(d;d)\times U_\kappa\C(\bd c;d)
\ar[ur]_-{\Gamma}
}
\]
commutes.
\end{proposition}

\begin{proof}
Again we unpack the counterclockwise direction using the definitions, and we get
\[
\xymatrix{
\{1\}\times\C(\kappabar\bd c,d)
\ar[d]^-{1\times\id}
\\\C(\kappabar d,d)\times\C(\kappabar\bd c,d)
\ar[d]^-{\id\times\kappa_1d}
\\\C(\kappabar d,d)\times\C(\kappa_1d(\kappabar\bd c),\kappabar d)
\ar[d]^-{1\times\C(\phi(d,\bd c),1)}
\\\C(\kappabar d,d)\times\C(\kappabar\bd c,\kappabar d)
\ar[d]^-{\circ}
\\\C(\kappabar\bd c,d).
}
\]
Looking at what happens before the final composition, the initial $\{1\}$ gets sent to the canonical
map $\omega(d):\kappabar d\to d$ in $\C(\kappabar d,d)$.  We can delay the composition in order to 
insert maps induced by $\omega(d)$ and its inverse as follows, without affecting the outcome:
\[
\xymatrix{
\{1\}\times\C(\kappabar\bd c,d)
\ar[d]^-{1\times\id}
\\\C(\kappabar d,d)\times\C(\kappabar\bd c,d)
\ar[d]^-{\id\times\kappa_1d}
\\\C(\kappabar d,d)\times\C(\kappa_1d(\kappabar\bd c),\kappabar d)
\ar[d]^-{1\times\C(\phi(d,\bd c),1)}
\\\C(\kappabar d,d)\times\C(\kappabar\bd c,\kappabar d)
\ar[d]^-{\C(\omega(d)^{-1},1)\times\C(1,\omega(d))}
\\\C(d,d)\times\C(\kappabar\bd c,d)
\ar[d]^-{\circ}
\\\C(\kappabar\bd c,d).
}
\]
But since the first map sends 1 to $\omega(d)$,
now the initial $\{1\}$ ends up at $\id_d\in\C(d,d)$, so the issue is making sure that the second factor
$\C(\kappabar\bd c,d)$ has its identity as the total map.  
Tracing a typical element $f\in\C(\kappabar\bd c,d)$, so $f:\kappabar\bd c\to d$, we see that the issue
is whether the following square commutes:
\[
\xymatrix@C+10pt{
\kappa_1d(\kappabar\bd c)
\ar[r]^-{\kappa_1d(f)}
&\kappabar d
\ar[d]^-{\omega(d)}
\\\kappabar\bd c
\ar[u]^-{\phi(d,\bd c)}
\ar[r]_-{f}
&d.
}
\]
Now $\phi(d,\bd c)$ arises from the unique isomorphism $\kappa_n\bd c\to\gamma(\kappa_1d;\kappa_n\bd c)$ in $Y(n)$,
but this can be expressed as 
\[
\gamma(\omega(d)^{-1};\kappa_n\bd c):\kappa_n\bd c=\gamma(1;\kappa_n\bd c)\to\gamma(\kappa_1d;\kappa_n\bd c)
\]
by the uniqueness of morphisms
in $Y(n)$.  This means that $\phi(d,\bd c)$ is really just $\omega(d)^{-1}$ applied to the object $\kappabar\bd c$,
so we might as well express the diagram as
\[
\xymatrix@C+10pt{
\kappa_1d(\kappabar\bd c)
\ar[r]^-{\kappa_1d(f)}
&\kappabar d
\ar[d]^-{\omega(d)}
\\\kappabar\bd c
\ar[u]^-{\omega(d)^{-1}}
\ar[r]_-{f}
&d.
}
\]
Since $\omega(d)$ is a natural isomorphism, the square commutes by naturality.  The second unit diagram therefore commutes.
\end{proof}

We next verify the commutativity of the equivariance diagram labeled (4) on p.\ 169 in \cite{EM1}.

\begin{theorem}
Let $\sigma\in\Sigma_n$.  Then
the following equivariance diagram in $U_\kappa\C$ commutes:
\[
\xymatrix{
U_\kappa\C(\bd c;d)\times\prod_{s=1}^n U_\kappa\C(\bd b_s;c_s)
\ar[r]^-{\Gamma}
\ar[d]_-{\sigma^*\times\sigma^{-1}}
&U_\kappa\C(\odot_s\bd b_s;d)
\ar[d]^-{\sigma\br{j_{\sigma(1)},\dots,j_{\sigma(n)}}^*}
\\U_\kappa\C(\sigma^{-1}\bd c,d)\times\prod_{s=1}^n U_\kappa\C(\bd b_{\sigma(s)},c_{\sigma(s)})
\ar[r]_-{\Gamma}
&U_\kappa\C(\odot_s\bd b_{\sigma(s)},d).
}
\]
\end{theorem}

\begin{proof}
The counterclockwise direction unpacks as follows:
\[
\xymatrix{
\C(\kappabar\bd c,d)\times\C^n(\br{\kappabar\bd b_s},\bd c)
\ar[d]^-{\sigma^*\times\sigma^{-1}}
\\\C(\kappabar(\sigma^{-1}\bd c),d)\times\C^n(\br{\kappabar\bd b_{\sigma(s)}},\sigma^{-1}\bd c)  
\ar[d]^-{1\times\kappa_n(\sigma^{-1}\bd c)}
\\\C(\kappabar(\sigma^{-1}\bd c),d)\times\C(\kappa_n(\sigma^{-1}\bd c)\br{\kappabar\bd b_{\sigma(s)}},\kappabar(\sigma^{-1}\bd c))
\ar[d]^-{\circ}
\\\C(\kappa_n(\sigma^{-1}\bd c)\br{\kappabar\bd b_{\sigma(s)}},d)
\ar[d]^-{\C(\phi(\sigma^{-1}\bd c,\br{\bd b_{\sigma(s)}}),1)}
\\\C(\kappabar(\odot_s\bd b_{\sigma(s)}),d).
}
\]

The clockwise direction unpacks as follows:
\[
\xymatrix{
\C(\kappabar\bd c,d)\times\C^n(\br{\kappabar\bd b_s},\bd c)
\ar[d]^-{1\times\kappa_n\bd c}
\\\C(\kappabar\bd c,d)\times\C(\kappa_n\bd c\br{\kappabar\bd b_s},\kappabar\bd c)
\ar[d]^-{\circ}
\\\C(\kappa_n\bd c\br{\kappabar\bd b_s},d)
\ar[d]^-{\C(\phi(\bd c,\br{\bd b_s}),1)}
\\\C(\kappabar(\odot_s\bd b_s),d)
\ar[d]^-{=}
\\\C([\kappa_j(\odot_s\bd b_s)\cdot\sigma\br{j_{\sigma(1)},\dots,j_{\sigma(n)}}](\odot_s\bd b_{\sigma(s)}),d)
\ar[d]^-{\C(\theta(\odot_s\bd b_s,\sigma\br{j_{\sigma(1)},\dots,j_{\sigma(n)}}),1)}
\\\C(\kappabar(\odot_s\bd b_{\sigma(s)}),d).
}
\]

To begin connecting the two directions, we first decompose the second arrow in the last display, labeled $\circ$ for composition,
as follows: 
\[
\xymatrix{
\C(\kappabar\bd c,d)\times\C(\kappa_n\bd c\br{\kappabar\bd b_s},\kappabar\bd c)
\ar[d]^-{=}
\\\C([\kappa_n\bd c\cdot\sigma]\sigma^{-1}\bd c,d)\times\C(\kappa_n\bd c\br{\kappabar\bd b_s},[\kappa_n\bd c\cdot\sigma]\sigma^{-1}\bd c)
\ar[d]^-{\C(\theta(\bd c,\sigma),1)\times\C(1,\theta(\bd c,\sigma)^{-1})}
\\\C(\kappabar(\sigma^{-1}\bd c),d)\times\C(\kappa_n\bd c\br{\kappabar\bd b_s},\kappabar(\sigma^{-1}\bd c))
\ar[d]^-{\circ}
\\\C(\kappa_n\bd c\br{\kappabar\bd b_s},d).
}
\]
Since we are composing in one factor with $\theta(\bd c,\sigma)$ and in the other with its inverse, the composition map remains
the same.

Now we connect the two directions by means of the following diagram, which we claim commutes, where the map $p$ is defined below:  
\[
\xymatrix@C-80pt@R-12pt{
C(\kappabar\bd c,d)\times\C^n(\br{\kappabar\bd b_s},\bd c)
\ar[dd]_-{\sigma^*\times\sigma^{-1}}
\ar[dr]^-{1\times\kappa_n\bd c}
\\&\C(\kappabar\bd c,d)\times\C(\kappa_n\bd c\br{\kappabar\bd b_s},\kappabar\bd c)
\ar[dd]^-{=}
\\\C(\kappabar(\sigma^{-1}\bd c),d)\times\C^n(\br{\kappabar\bd b_{\sigma(s)}},\sigma^{-1}\bd c) 
\ar[dd]_-{1\times\kappa_n(\sigma^{-1}\bd c)}
\\&\C([\kappa_n\bd c\cdot\sigma]\sigma^{-1}\bd c,d)\times\C(\kappa_n\bd c\br{\kappabar\bd b_s},[\kappa_n\bd c\cdot\sigma]\sigma^{-1}\bd c)
\ar[dd]^-{\C(\theta(\bd c,\sigma),1)\times\C(1,\theta(\bd c,\sigma)^{-1})}
\\\C(\kappabar(\sigma^{-1}\bd c),d)\times\C(\kappa_n(\sigma^{-1}\bd c)\br{\kappabar\bd b_{\sigma(s)}},\kappabar(\sigma^{-1}\bd c))
\ar[dr]_-{1\times\C(p,1)\,\,\,\,\,\,\,\,\,\,\,\,}
\ar[dd]_-{\circ}
\\&\C(\kappabar(\sigma^{-1}\bd c),d)\times\C(\kappa_n\bd c\br{\kappabar\bd b_s},\kappabar(\sigma^{-1}\bd c))
\ar[d]^-{\circ}
\\\C(\kappa_n(\sigma^{-1}\bd c)\br{\kappabar\bd b_{\sigma(s)}},d)
\ar[r]^-{\C(p,1)}
\ar[ddd]_-{\C(\phi(\sigma^{-1}\bd c,\br{\bd b_{\sigma(s)}}),1)}
&\C(\kappa_n\bd c\br{\kappabar\bd b_s},d)
\ar[d]^-{\C(\phi(\bd c,\br{\bd b_s}),1)}
\\\relax
&\C(\kappabar(\odot_s\bd b_s),d)
\ar[d]^-{=}
\\&\C([\kappa_j(\odot_s\bd b_s)\cdot\sigma\br{j_{\sigma(1)},\dots,j_{\sigma(n)}}](\odot_s\bd b_{\sigma(s)}),d)
\ar[dl]^-{\,\,\,\,\,\,\,\,\,\,\,\,\,\,\,\,\,\,\,\,\,\,\,\,\,\,\,\,\,\,\,\,\,\,\,\C(\theta(\odot_s\bd b_s,\sigma\br{j_{\sigma(1)},\dots,j_{\sigma(n)}}),1)}
\\\C(\kappabar(\odot_s\bd b_{\sigma(s)}),d)
}
\]
Here $p$ is induced by the unique
isomorphism
\[
\gamma(\kappa_n\bd c;\br{\kappa_{j_s}\bd b_s})\cdot\sigma\br{j_{\sigma(1)},\dots,j_{\sigma(n)}}
\to\gamma(\kappa_n(\sigma^{-1}\bd c);\br{\kappa_{j_{\sigma(s)}}\bd b_s})
\]
in $Y(j)$, which in turn arises as the composite
\[
\xymatrix{
\gamma(\kappa_n\bd c;\br{\kappa_{j_s}\bd b_s})\cdot\sigma\br{j_{\sigma(1)},\dots,j_{\sigma(n)}}
\ar[d]^-{=}
\\\gamma(\kappa_n\bd c\cdot\sigma\cdot\sigma^{-1};\br{\kappa_{j_s}\bd b_s})\cdot\sigma\br{j_{\sigma(1)},\dots,j_{\sigma(n)}}
\ar[d]^-{=}
\\\gamma(\kappa_n\bd c\cdot\sigma;\br{\kappa_{j_{\sigma(s)}}\bd b_{\sigma(s)}})\cdot\sigma^{-1}\br{j_1,\dots,j_n}\cdot\sigma\br{j_{\sigma(1)},\dots,j_{\sigma(n)}}
\ar[d]^-{=}
\\\gamma(\kappa_n\bd c\cdot\sigma;\br{\kappa_{j_{\sigma(s)}}\bd b_{\sigma(s)}})
\ar[d]^-{\gamma(\theta(\bd c,\sigma)^{-1};1)}
\\\gamma(\kappa_n(\sigma^{-1}\bd c);\br{\kappa_{j_{\sigma(s)}}\bd b_{\sigma(s)}}),
}
\]
where the second equality arises from the first equivariance formula for an operad in \cite{G}, Definition 1.1(c) on page 2.
We are therefore also justified in writing $p$ as $\theta(\bd c,\sigma)^{-1}$.

The bottom pentagon commutes, since each map is induced by an isomorphism in $Y(j)$, where all diagrams commute.
The middle (distorted) square commutes since the horizontal arrows are given by the same map $p$.  This leaves
the top hexagon to check, and we can do so on each factor separately,  
since the hexagon is actually the product of two separate hexagons.
On the first factor $\C(\kappabar\bd c, d)$, the
counterclockwise direction is given by the map $\sigma^*$.  But this is defined to be induced by $\theta(\bd c, \sigma)$,
which is what induces the clockwise direction.  The top portion therefore commutes on the first factor, thus reducing the question
to its commutativity on the second factor.  This is captured in the perimeter of the following diagram:
\[
\xymatrix{
&\C^n(\br{\kappabar\bd b_s},\bd c)
\ar[dl]_-{\sigma^{-1}}
\ar[dr]^-{\kappa_n\bd c}
\\\C^n(\br{\kappabar\bd b_{\sigma(s)}},\sigma^{-1}\bd c)
\ar[d]_-{\kappa_n(\sigma^{-1}\bd c)}
\ar[rr]^-{\kappa_n\bd c\cdot\sigma}
&&\C(\kappa_n\bd c\br{\kappabar\bd b_s},\kappabar\bd c)
\ar[d]^-{\C(1,\theta(\bd c,\sigma)^{-1})}
\\\C(\kappa_n(\sigma^{-1}\bd c)\br{\kappabar\bd b_{\sigma(s)}},\kappabar(\sigma^{-1}\bd c))
\ar[rr]_-{\C(p,1)}
&&\C(\kappa_n\bd c\br{\kappabar\bd b_s},\kappabar(\sigma^{-1}\bd c)).
}
\]
The triangle on top commutes because $\kappa_n\bd c\cdot\sigma$ is just the composition of $\kappa_n\bd c$ with the
permutation given by $\sigma$, so composing with $\sigma^{-1}$ just gives $\kappa_n\bd c$.  This reduces us to verifying
commutativity of the bottom square.  But we can rewrite the bottom square as follows:
\[
\xymatrix@C-50pt{
&\C^n(\br{\kappabar\bd b_{\sigma(s)}},\sigma^{-1}\bd c)
\ar[dl]_-{\kappa_n(\sigma^{-1}\bd c)}
\ar[dr]^-{\kappa_n\bd c\cdot\sigma}
\\\C(\kappa_n(\sigma^{-1}\bd c)\br{\kappabar\bd b_{\sigma(s)}},\kappabar(\sigma^{-1}\bd c))
\ar[dr]_-{\C(\theta(\bd c,\sigma)^{-1},1)\,\,\,\,\,}
\ar[rr]^-{\C(\theta(\bd c,\sigma)^{-1},\theta(\bd c,\sigma))}
&&\C(\kappa_n\bd c\cdot\sigma\br{\kappabar\bd b_{\sigma(s)}},\kappa_n\cdot\sigma(\sigma^{-1}\bd c))
\ar[dl]^-{\,\,\,\,\,\C(1,\theta(\bd c,\sigma)^{-1})}
\\&\C(\kappa_n\bd c\cdot\sigma\br{\kappabar\bd b_{\sigma(s)}},\kappabar(\sigma^{-1}\bd c))
\ar[d]^-{=}
\\&\C(\kappa_n\bd c\br{\kappabar\bd b_s},\kappabar(\sigma^{-1}\bd c)).
}
\]
The top triangle commutes by the naturality of $\theta(\bd c,\sigma)$, and the bottom triangle by inspection.  This
concludes the verification of the first equivariance diagram.
\end{proof}

The verification that $U_\kappa\C$ satisfies the requirements for a multicategory concludes with the commutativity
of the second equivariance diagram, labeled (4) on p.\ 169 of \cite{EM1}.  We suppose given permutations $\tau_s\in\Sigma_{j_s}$
for $1\le s\le n$.

\begin{proposition}
The following equivariance diagram commutes:
\[
\xymatrix{
U_\kappa\C(\bd c;d)\times\prod_{s=1}^n U_\kappa\C(\bd b_s;c_s)
\ar[r]^-{\Gamma}
\ar[d]_-{1\times\prod\tau_s^*}
&U_\kappa\C(\odot_s\bd b_s,d)
\ar[d]^-{(\oplus_s\tau_s)^*}
\\U_\kappa\C(\bd c;d)\times\prod_{s=1}^n U_\kappa\C(\tau_s^{-1}\bd b_s,c_s)
\ar[r]_-{\Gamma}
&U_\kappa\C(\odot_s\tau_s^{-1}\bd b_s;d).
}
\]
\end{proposition}

\begin{proof}
The clockwise direction unpacks as follows:
\[
\xymatrix{
\C(\kappabar\bd c,d)\times\C^n(\br{\kappabar\bd b_s},\bd c)
\ar[d]^-{1\times\kappa_n\bd c}
\\\C(\kappabar\bd c,d)\times\C(\kappa_n\bd c\br{\kappabar\bd b_s},\kappabar\bd c)
\ar[d]^-{\circ}
\\\C(\kappa_n\bd c\br{\kappabar\bd b_s},d)
\ar[d]^-{\C(\phi(\bd c,\br{\bd b_s}),1)}
\\\C(\kappabar(\odot_s\bd b_s),d)
\ar[d]^-{(\oplus_s\tau_s)^*}
\\\C(\kappabar(\odot_s\tau_s^{-1}\bd b_s),d).
}
\]
However, it will be convenient to rearrange this to delay the composition to the end, as follows:
\[
\xymatrix{
\C(\kappabar\bd c,d)\times\C^n(\br{\kappabar\bd b_s},\bd c)
\ar[d]^-{1\times\kappa_n\bd c}
\\\C(\kappabar\bd c,d)\times\C(\kappa_n\bd c\br{\kappabar\bd b_s},\kappabar\bd c)
\ar[d]^-{1\times\C(\phi(\bd c,\br{\bd b_s}),1)}
\\\C(\kappabar\bd c,d)\times\C(\kappabar(\odot_s\bd b_s),\kappabar\bd c)
\ar[d]^-{1\times(\oplus_s\tau_s)^*}
\\\C(\kappabar\bd c,d)\times\C(\kappabar(\odot_s\tau_s^{-1}\bd b_s),\kappabar\bd c)
\ar[d]^-{\circ}
\\\C(\kappabar(\odot_s\tau_s^{-1}\bd b_s),d).
}
\]

The counterclockwise direction unpacks as follows:
\[
\xymatrix{
\C(\kappabar\bd c,d)\times\C^n(\br{\kappabar\bd b_s},\bd c)
\ar[d]^-{1\times\prod_s\tau_s^*}
\\\C(\kappabar\bd c,d)\times\C^n(\br{\kappabar(\tau_s^{-1}\bd b_s)},\bd c)
\ar[d]^-{1\times\kappa_n\bd c}
\\\C(\kappabar\bd c,d)\times\C(\kappa_n\bd c\br{\kappabar(\tau_s^{-1}\bd b_s)},\kappabar\bd c)
\ar[d]^-{\circ}
\\\C(\kappa_n\bd c\br{\kappabar(\tau_s^{-1}\bd b_s)},d)
\ar[d]^-{\C(\phi(\bd c,\br{\tau_s^{-1}\bd b_s}),1)}
\\\C(\kappabar(\odot_s\tau_s^{-1}\bd b_s),d).
}
\]
But again, it will be convenient to rewrite this to delay the composition to the end, as follows:
\[
\xymatrix{
\C(\kappabar\bd c,d)\times\C^n(\br{\kappabar\bd b_s},\bd c)
\ar[d]^-{1\times\prod_s\tau_s^*}
\\\C(\kappabar\bd c,d)\times\C^n(\br{\kappabar(\tau_s^{-1}\bd b_s)},\bd c)
\ar[d]^-{1\times\kappa_n\bd c}
\\\C(\kappabar\bd c,d)\times\C(\kappa_n\bd c\br{\kappabar(\tau_s^{-1}\bd b_s)},\kappabar\bd c)
\ar[d]^-{1\times\C(\phi(\bd c,\br{\tau_s^{-1}\bd b_s}),1)}
\\\C(\kappabar\bd c, d)\times\C(\kappabar(\odot_s\tau_s^{-1}\bd b_s),\kappabar\bd c)
\ar[d]^-{\circ}
\\\C(\kappabar(\odot_s\tau_s^{-1}\bd b_s),d).
}
\]
We can connect the two directions, which form the perimeter of the following diagram, in which
the hexagon is actually the product of two separate hexagons:
\[
\xymatrix@C-60pt{
&\C(\kappabar\bd c,d)\times\C^n(\br{\kappabar\bd b_s},\bd c)
\ar[dl]_-{1\times\prod_s\tau_s^*}
\ar[dr]^-{1\times\kappa_n\bd c}
\\\C(\kappabar\bd c,d)\times\C^n(\br{\kappabar(\tau_s^{-1}\bd b_s)},\bd c)
\ar[d]_-{1\times\kappa_n\bd c}
&&\C(\kappabar\bd c,d)\times\C(\kappa_n\bd c\br{\kappabar\bd b_s},\kappabar\bd c)
\ar[dll]_-{1\times\C(\kappa_n\bd c\br{\theta(\bd b_s,\tau_s)},1)\,\,\,\,\,\,\,\,\,\,\,\,\,\,}
\ar[d]^-{1\times\C(\phi(\bd c,\br{\bd b_s}),1)}
\\\C(\kappabar\bd c,d)\times\C(\kappa_n\bd c\br{\kappabar(\tau_s^{-1}\bd b_s)},\kappabar\bd c)
\ar[dr]_-{1\times\C(\phi(\bd c,\br{\tau_s^{-1}\bd s}),1)\,\,\,\,\,\,\,\,\,\,\,\,}
&&\C(\kappabar\bd c,d)\times\C(\kappabar(\odot_s\bd b_s),\kappabar\bd c)
\ar[dl]^-{1\times(\oplus_s\tau_s)^*}
\\&\C(\kappabar\bd c,d)\times\C(\kappabar(\odot_s\tau_s^{-1}\bd b_s),\kappabar\bd c)
\ar[d]^-{\circ}
\\&\C(\kappabar(\odot_s\tau_s^{-1}\bd b_s),d).
}
\]
Since nothing happens to the first factor, $\C(\kappabar\bd c, d)$, until the end composition, we may ignore it,
and concentrate on the second factor, $\C^n(\br{\kappabar\bd b_s},\bd c)$.  The top square in the diagram
becomes
\[
\xymatrix{
\C^n(\br{\kappabar\bd b_s},\bd c)
\ar[r]^-{\kappa_n\bd c}
\ar[d]_-{\prod_s\tau_s^*}
&\C(\kappa_n\bd c\br{\kappabar\bd b_s},\kappabar\bd c)
\ar[d]^-{\C(\kappa_n\bd c\br{\theta(\bd b_s,\tau_s)},1)}
\\\C^n(\br{\kappabar(\tau^{-1}\bd b_s)},\bd c)
\ar[r]_-{\kappa_n\bd c}
&\C(\kappa_n\bd c\br{\kappabar(\tau^{-1}\bd b_s)},\kappabar\bd c).
}
\]
But the maps $\tau_s^*$ are induced by $\theta(\bd b_s,\tau_s)$, so we can rewrite the
left vertical arrow as $\C^n(\br{\theta(\bd b_s,\tau_s)},1)$, and the diagram now commutes
by functoriality of $\kappa_n\bd c$.

The lower square in our desired diagram becomes
\[
\xymatrix@C+50pt{
\C(\kappa_n\bd c\br{\kappabar\bd b_s},\kappabar\bd c)
\ar[r]^-{\C(\phi(\bd c,\br{\bd b_s}),1)}
\ar[d]_-{\C(\kappa_n\bd c\br{\theta(\bd b_s,\tau_s)},1)}
&\C(\kappabar(\odot_s\bd b_s),\kappabar\bd c)
\ar[d]^-{(\oplus_s\tau_s)^*}
\\\C(\kappa_n\bd c\br{\kappabar(\tau^{-1}\bd b_s)},\kappabar\bd c)
\ar[r]_-{\C(\phi(\bd c,\br{\tau_s^{-1}\bd b_s)},1)}
&\C(\kappabar(\tau_s^{-1}\bd b_s),\kappabar\bd c).
}
\]
Since all the arrows in this diagram are induced from morphisms in $Y(n)$, where all diagrams commute,
this diagram also commutes.  This completes the verification that the second equivariance diagram commutes,
and therefore that $U_\kappa\C$ satisfies the requirements for a multicategory.
\end{proof}

\end{document}